\documentclass[11pt]{article}

\title{\bf Non-classification of Cartan subalgebras for a class of von Neumann algebras}
\author{\textsc{Pieter Spaas\thanks{The author was partially supported by NSF Career Grant DMS \#1253402}}}
\date{}
\newcommand{\Addresses}{{
		\bigskip
		\footnotesize
		
		\textsc{Department of Mathematics; University of California San Diego, CA 92093 (United States).}\par\nopagebreak
		\textit{E-mail address:} \texttt{pspaas@ucsd.edu}

}}

\usepackage[top=2.5cm, bottom=2.5cm, left=2.4cm, right=2.4cm]{geometry}

\usepackage{amsthm,amsmath,amsfonts,mathtools,thm-restate,amssymb}
\usepackage{graphicx,subcaption}
\usepackage[all,cmtip]{xy}
\usepackage{relsize,enumitem}
\usepackage{cancel}

\usepackage[backend=bibtex,style=ieee-alphabetic]{biblatex}
\usepackage[nottoc]{tocbibind} 
\usepackage[colorlinks=false,pdfborder={0 0 0}]{hyperref}
\stdpunctuation 
\renewbibmacro*{bbx:savehash}{} 

\usepackage{thmtools}

\declaretheorem[style=plain,numberwithin=section]{theorem}
\declaretheorem[style=plain,sibling=theorem]{lemma}
\declaretheorem[style=plain,sibling=theorem]{proposition}
\declaretheorem[style=plain,sibling=theorem]{corollary}

\declaretheorem[style=plain,title=Theorem]{thm}

\declaretheorem[style=plain,title=Corollary,sibling=thm]{cor}

\declaretheorem[style=definition,qed=$\blacktriangleleft$,sibling=theorem]{definition}
\declaretheorem[style=definition,sibling=theorem]{example}

\declaretheorem[style=remark,numbered=no]{remark}
\declaretheorem[style=definition,numbered=no]{terminology}
\declaretheorem[style=definition,numbered=no]{convention}
\declaretheorem[style=remark]{Claim}
\declaretheorem[style=remark,numbered=no]{claim}

\newenvironment{proofofclaim}{\textit{Proof of the claim.}}{\hfill$\diamond$}

\newcommand{\cA}{\mathcal{A}}
\newcommand{\bB}{\mathbb{B}}\newcommand{\cB}{\mathcal{B}}
 \newcommand{\C}{\mathbb{C}}\newcommand{\cC}{\mathcal{C}}

\newcommand{\bF}{\mathbb{F}}

\newcommand{\cH}{\mathcal{H}}

\newcommand{\cM}{\mathcal{M}}
 \newcommand{\N}{\mathbb{N}}\newcommand{\cN}{\mathcal{N}}
\newcommand{\cO}{\mathcal{O}}

 \newcommand{\Q}{\mathbb{Q}}
 \newcommand{\R}{\mathbb{R}}\newcommand{\cR}{\mathcal{R}}

\newcommand{\cU}{\mathcal{U}}

 \newcommand{\Z}{\mathbb{Z}}\newcommand{\cZ}{\mathcal{Z}}

\usepackage{bbm}
\newcommand{\een}{\mathbbm{1}}

\newcommand{\al}{\alpha}

\newcommand{\be}{\beta}
\newcommand{\gam}{\gamma}
\newcommand{\Gam}{\Gamma}
\newcommand{\del}{\delta}
\newcommand{\Del}{\Delta}
\newcommand{\eps}{\varepsilon}
\newcommand{\lam}{\lambda}
\newcommand{\Lam}{\Lambda}

\newcommand{\vphi}{\varphi}

\newcommand{\norm}[1]{\left\|#1\right\|}
\newcommand{\abs}[1]{\left|#1\right|}
\newcommand{\ip}[2]{\langle #1,#2 \rangle}
\newcommand{\sub}{\subseteq}
\newcommand{\super}{\supseteq}
\newcommand{\ot}{\otimes}
\newcommand{\otb}{\bar{\otimes}}

\newcommand{\rarrow}{\rightarrow}
\newcommand{\inv}{^{-1}}
\renewcommand{\iff}{\Leftrightarrow}

\DeclareSymbolFont{Symbols}{OMS}{cmsy}{m}{n}
\DeclareMathSymbol{\Emptyset}{\mathord}{Symbols}{"3B}

\newcommand{\act}{\curvearrowright}
\newcommand{\cross}{\rtimes}

\newcommand{\dint}{\int^\oplus}
\newcommand{\emb}{\prec}
\newcommand{\II}{II$_1$ }

\newcommand{\Aut}{\operatorname{Aut}}

\newcommand{\Hom}{\operatorname{Hom}}
\newcommand{\Map}{\operatorname{Map}}
\newcommand{\ClSub}{\operatorname{ClSub}}

\newcommand{\Ad}{\operatorname{Ad}}
\newcommand{\id}{\operatorname{id}}

\newcommand{\ctr}{\operatorname{ctr}}
\newcommand{\vNa}{\operatorname{vNa}}
\newcommand{\diag}{\operatorname{diag}}

\newcommand{\Cartan}{\operatorname{Cartan}}

\newcommand{\LinfN}{L^\infty(X)\otb N}
\newcommand{\Linf}{L^\infty(X)}

\numberwithin{equation}{section}

\newcommand{\emphind}[1]{\emph{#1}\index{#1}}

\begin{document}

\maketitle

\begin{quote}
	\textsc{Abstract.} We study the complexity of the classification problem for Cartan subalgebras in von Neumann algebras. We construct a large family of II$_1$ factors whose Cartan subalgebras up to unitary conjugacy are not classifiable by countable structures, providing the first such examples. Additionally, we construct examples of II$_1$ factors whose Cartan subalgebras up to conjugacy by an automorphism are not classifiable by countable structures. Finally, we show directly that the Cartan subalgebras of the hyperfinite II$_1$ factor up to unitary conjugacy are not classifiable by countable structures, and deduce that the same holds for any McDuff II$_1$ factor with at least one Cartan subalgebra.
\end{quote}

\section*{Introduction}

Since the first paper of Murray and von Neumann \cite{MvN36} the study of von Neumann algebras has been closely related to the study of group actions on measure spaces. The \textit{group measure space construction} introduced in their paper associates to every free ergodic probability measure preserving (pmp) action $\Gam\act (X,\mu)$ of a countable group $\Gam$ on a probability measure space $(X,\mu)$ a crossed product von Neumann algebra $\Linf\cross\Gam$, which turns out to be a factor of type II$_1$. The classification of these group measure space von Neumann algebras is in general a very hard problem. Nevertheless a lot of progress has been made on several fronts. 

\vspace{5pt}

One recurring theme in all known classification results is the special role played by the so-called \text{Cartan subalgebras} - maximal abelian subalgebras whose normalizer generates the \II factor. Given a free ergodic pmp action $\Gam\act X$, the group measure space \II factor $\Linf\cross\Gam$ always contains $\Linf$ as a Cartan subalgebra. Moreover, Singer established in \cite{Si55} that two free ergodic pmp actions $\Gam\act X$ and $\Lam\act Y$ are \textit{orbit equivalent} if and only if there exists an isomorphism $\theta:L^\infty(X)\cross\Gam\rarrow L^\infty(Y)\cross\Lam$ such that $\theta(L^\infty(X))=L^\infty(Y)$. This result both gives an immediate link with measured group theory and highlights the importance of understanding the Cartan subalgebras of a given \II factor. For instance, if one can show that certain classes of group measure space \II factors have a unique Cartan subalgebra (up to conjugacy by an automorphism), their classification up to isomorphism reduces to the classification of the corresponding actions up to orbit equivalence. 

\vspace{5pt}

If the acting group is amenable the group measure space construction will always give rise to the hyperfinite \II factor $R$ by Connes' famous theorem \cite{Co76}. Moreover it was shown in \cite{CFW81} that all Cartan subalgebras of $R$ are conjugate by an automorphism of $R$. Hence all free ergodic pmp actions of all amenable groups are orbit equivalent and besides the group being amenable, one cannot recover any more information about the group nor the action from the group measure space \II factor. 

\vspace{5pt}

On the other hand, for non-amenable groups there is a wide spectrum of different \II factors appearing and rigidity results start showing up. During the last 15 years, Popa's \textit{deformation/rigidity} theory has led to considerable progress in the classification of group measure space \II factors, see for instance the surveys \cite{Po07,Va10,Io12b}. In particular, several uniqueness results for Cartan subalgebras have been established. The first such result was obtained by Ozawa and Popa in \cite{OP07}. They proved that $\Linf$ is the unique Cartan subalgebra up to unitary conjugacy inside $\Linf\cross\bF_n$ where $\bF_n$ is a non-abelian free group and $\bF_n\act X$ is a free ergodic \textit{profinite} pmp action. The class of groups whose profinite actions give \II factors with a unique Cartan subalgebra was subsequently extended in \cite{OP08,CS11}. Later the condition of profiniteness was removed by Popa and Vaes, who showed in \cite{PV11} that \textit{any} free ergodic pmp action of a non-abelian free group gives rise to a \II factor with a unique Cartan subalgebra up to unitary conjugacy. Hereafter, additional uniqueness results were obtained in (among others) \cite{PV12,Io12a,CIK13}.

\vspace{5pt}

Of course not every \II factor has a unique Cartan subalgebra. The first \II factor $M$ containing at least two Cartan subalgebras that are not conjugate by an automorphism of $M$ was constructed in \cite{CJ82}. By now there are more examples known (see for instance \cite{OP08,PV09}). Nevertheless it is worth mentioning that at the moment, as soon as uniqueness fails, we do not have any examples of \II factors for which we can describe all Cartan subalgebras up to unitary conjugacy in a satisfactory way. Some progress in this direction was made by Krogager and Vaes in \cite{KV15}, where they construct \II factors for which they can describe all \textit{group measure space Cartan subalgebras}, i.e. Cartan subalgebras $A$ arising from a decomposition of the \II factor $M$ as a crossed product $M\cong A\cross\Gam$ for some countable group $\Gam$. In particular they construct a \II factor with exactly two group measure space Cartan subalgebras up to unitary conjugacy.

\vspace{5pt}

Another result in the non-uniqueness direction was obtained by Speelman and Vaes in \cite[Theorem~2]{SV11}. They construct a class of \II factors for which the relations of unitary conjugacy and conjugacy by a (stable) automorphism on Cartan subalgebras are not smooth (or not concretely classifiable, see also Definition~\ref{def:smooth}).

\vspace{5pt}

In this paper we continue in this direction by exploring further the complexity, in the sense of descriptive set theory, of the classification problem for Cartan subalgebras. We will consider both the equivalence relations of unitary conjugacy and of conjugation by an automorphism on the space of Cartan subalgebras of a family of \II factors. Using different techniques, we will provide the first examples of \II factors whose Cartan subalgebras are not classifiable by countable structures for either notion of equivalence. Hereby we confirm the statement in \cite[Remark~14]{SV11} expressing the belief that such \II factors should exist. We refer to section~\ref{sec:classctbstruct} for the definition of classifiability by countable structures. Intuitively this means we cannot construct complete invariants out of countable (or discrete) structures such as countable groups, graphs, etc. and that these equivalence relations are \textquotedblleft beyond $S_\infty$-actions".

\subsection*{Statement of the main results}

Let $\Gam$ be a relatively strongly solid group (see Definition~\ref{def:rss}, this class includes all non-abelian free groups and more generally all non-elementary hyperbolic groups by \cite{PV11,PV12}). Let further $(X,\mu)$ be a standard probability space, $\Gam\act X$ a free ergodic pmp action and $N$ an arbitrary \II factor. Consider
\[
\cM:=(\Linf\cross\Gam)\otb N.
\]
It turns out that the structure of $\Cartan(\cM)$, the space of Cartan subalgebras of $\cM$, is completely determined by the structure of $\Cartan(N)$ and the orbits of the action $\Gam\act X$. More precisely, we can write $\cM=(\LinfN)\cross\Gam$ where $\Gam$ acts trivially on $N$. Taking the (in this case constant) integral decomposition of $\LinfN$ over its center $\Linf$, we can write $\LinfN=\dint_XN\,d\mu(x)$. We refer to section~\ref{sec:directint} where we collected the basic definitions and properties of direct integrals. The main structural result on the Cartan subalgebras of $\cM$ is the following theorem, which will be proved in section~\ref{sec:structure}.

\begin{restatable}{thm}{thmstructure}\label{thm:structure}
	Let $\Gam\in\cC_{rss}$ and $N$ be a \II factor. Suppose $(X,\mu)$ is a standard probability space and $\Gam\act X$ is a free ergodic pmp action. Consider $\cM:=(\Linf\cross\Gam)\otb N$ and let $A\sub\cM$ be a subalgebra. Then the following are equivalent:
	\begin{enumerate}
		\item $A$ is a Cartan subalgebra of $\cM$,
		\item $A$ is unitarily conjugate to a subalgebra $B$ of the form $B=\dint_X B_x\,d\mu(x) \sub L^\infty(X)\otb N$ satisfying
		\begin{itemize}
			\item $B_x$ is a Cartan subalgebra of $N$ for almost every $x$,
			\item For every $g\in \Gam$, $B_x$ is unitarily conjugate to $B_{gx}$ inside $N$ for almost every $x$.
		\end{itemize} 
	\end{enumerate}
\end{restatable}

Moreover it will follow from the proof of Theorem~\ref{thm:structure} that two Cartan subalgebras $A,B$ of $\cM$ contained in $\LinfN$ are unitarily conjugate if and only if $A_x$ is unitarily conjugate to $B_x$ for almost every $x\in X$, where we wrote $A=\dint_X A_x\,dx$ and $B=\dint_X B_x\,dx$. It turns out that this is in general not the case anymore for conjugacy by an automorphism. However, in specific cases, a similar result will hold (see section~\ref{sec:Aut}). We also note that the assumption $\Gam\in\cC_{rss}$ is only needed for the proof of $(1)\Rightarrow (2)$. In particular, for any countable group $\Gam$, every subalgebra $B$ as in $(2)$ of Theorem~\ref{thm:structure} will be a Cartan subalgebra of $(L^\infty(X)\cross\Gam)\otb N$.

One can also derive the following corollary in case $N$ has at most one Cartan subalgebra up to unitary conjugacy.

\begin{corollary}
	If $N$ has either no Cartan subalgebras or a unique Cartan subalgebra up to unitary conjugacy, then the same holds for $\cM=(\Linf\cross\Gam)\otb N$ as in Theorem~\ref{thm:structure}.
\end{corollary}

For any Polish group $G$ acting continuously on a Polish space $Y$ we will write $\cR(G\act Y)$ for its corresponding orbit equivalence relation. The proof of theorem~\ref{thm:structure} will lead to a criterion guaranteeing that the Cartan subalgebras of $\cM$ are not classifiable by countable structures. This will allow us to prove our main result below in section~\ref{sec:nonclass}. Note that we don't need the assumption $\Gam\in\cC_{rss}$ here, since we only use the part of Theorem~\ref{thm:structure} which doesn't need this assumption, namely the proof that $(2)\Rightarrow (1)$.

\begin{restatable}{thm}{thmnonclass}\label{thm:nonclass}
	Let $\Gam$ be a countable group and $N$ be a \II factor. Suppose $(X,\mu)$ is a standard probability space and $\Gam\act X$ is a free ergodic pmp action that is not strongly ergodic. Assume furthermore that $\cR(\cU(N)\act\Cartan(N))$ is not smooth. Then the Cartan subalgebras of $\cM:=(\Linf\cross\Gam)\otb N$ up to unitary conjugacy are not classifiable by countable structures.
\end{restatable}

Together with the results from \cite{SV11} this gives the first concrete family of (non-hyperfinite) \II factors whose Cartan subalgebras up to unitary conjugacy are not classifiable by countable structures. Indeed, one can take any countable group $\Gam$ admitting an ergodic non-strongly ergodic action (i.e. without property (T), cf. \cite{CW80}. In fact given any non-property (T) group $\Gam$, the generic action of $\Gam$ is ergodic but not strongly ergodic, see \cite[Corollary~12.4]{Kec10}.) One can then consider any free ergodic non-strongly ergodic pmp action of $\Gam$ on a standard probability space $(X,\mu)$. Taking $N$ to be the \II factor from \cite[Theorem~2(1)]{SV11}, $(L^\infty(X)\cross\Gam)\otb N$ will satisfy all assumptions in the above theorem.

In the course of the proof of Theorem~\ref{thm:nonclass} we will establish the following result, which appears to be of independent interest. We denote by $\Hom(\cR(\Gam\act X),E_0)$ the space of homomorphisms (see Definition~\ref{def:hom}) from $\cR(\Gam\act X)$ to $E_0$, where $E_0$ is the equivalence relation on $\{0,1\}^\N$ given by
\[
{\bf x}E_0{\bf y} \quad \Leftrightarrow \quad \exists N\in\N, \forall n\geq N: x_n=y_n.
\]
Also, for $\varphi,\psi\in \Hom(\cR(\Gam\act X),E_0)$ we let $\varphi\sim\psi$ if and only if $\varphi(x)E_0\psi(x)$ almost everywhere.

\begin{thm}\label{thm:nonstrongE_01}
	Let $\Gam$ be a countable group, $(X,\mu)$ a standard probability space and $\Gam\act X$ an ergodic pmp action that is not strongly ergodic. Then $(\Hom(\cR(\Gam\act X),E_0),\sim)$ is not classifiable by countable structures.
\end{thm}

Since an action that is strongly ergodic (also called $E_0$-ergodic) does not admit any nontrivial homomorphisms to $E_0$ (see \cite{JS87}, or \cite[Theorem~A2.2]{HK05} for a proof of exactly this statement), we get the following nice dichotomy for actions of countable groups.

\begin{corollary}
	Let $\Gam$ be a countable group, $(X,\mu)$ a standard probability space and $\Gam\act X$ an ergodic pmp action. Then $(\Hom(\cR(\Gam\act X),E_0),\sim)$ is either trivial or not classifiable by countable structures, depending on whether the action is strongly ergodic or not.
\end{corollary}

Moreover, it follows easily from Theorem~\ref{thm:structure} and Theorem~\ref{thm:nonclass} that the \II factors involved there actually satisfy the following dichotomy property.

\begin{restatable}{thm}{thmdichotomy}\label{thm:dichotomy}
	Let $\Gam\in\cC_{rss}$, $(X,\mu)$ be a standard probability space and $N$ be a \II factor. Suppose $\Gam\act X$ is a free ergodic pmp action that is not strongly ergodic and consider $\cM:=(\Linf\cross\Gam)\otb N$. Then $\cR(\cU(\cM)\act\Cartan(\cM))$ is either smooth or not classifiable by countable structures. Moreover, the former holds if and only if $\cR(\cU(N)\act\Cartan(N))$ is smooth and this regardless of whether the action is strongly ergodic or not.
\end{restatable}

In section~\ref{sec:Aut} we will discuss the equivalence relation of conjugacy by an automorphism on Cartan subalgebras. Using slightly different methods it turns out that for specific choices of $N$, the same construction as in Theorems~\ref{thm:nonclass} and \ref{thm:dichotomy} also gives \II factors for which the Cartan subalgebras up to conjugacy by an automorphism are not classifiable by countable structures.

\begin{restatable}{thm}{thmnonclassaut}\label{thm:nonclassaut}
	Suppose $N$ is the \II factor constructed in \cite[Theorem~2(1)]{SV11}. Let $\Gam\in\cC_{rss}$, $(X,\mu)$ be a standard probability space and $\Gam\act X$ be a free ergodic pmp action that is not strongly ergodic. Then the equivalence relation of Cartan subalgebras of $\cM=(\Linf\cross \Gam)\otb N$ up to conjugacy by an automorphism is not classifiable by countable structures.
\end{restatable}

The proof of Theorem~\ref{thm:nonclassaut} relies more on the structure of $N$ than the proof of Theorem~\ref{thm:nonclass}. We will formulate the exact requirements on $N$ in Theorem~\ref{thm:nonclassaut2}. In particular, the proof of this result will allow us to get the above specific examples of \II factors for which the Cartan subalgebras up to conjugacy by an automorphism are not classifiable by countable structures, but it will not allow us to get (in this way) a general dichotomy result as in Theorem~\ref{thm:dichotomy}.

\vspace{5pt}

Note that all our results we mentioned so far concern non-hyperfinite \II factors. As we already mentioned in the beginning of the introduction, the hyperfinite \II factor $R$ has a unique Cartan subalgebra up to conjugacy by an automorphism. On the other hand, J. Packer constructed in \cite[Theorem~4.4]{Pa85} an uncountable family of Cartan subalgebras of $R$ no two of which are unitarily conjugate. Combining her results with some turbulence results for cocycles from \cite{Kec10}, we will prove the following stronger statement in section~\ref{sec:R}.

\begin{restatable}{thm}{thmR}\label{thm:R}
	Cartan subalgebras of the hyperfinite \II factor $R$ up to unitary conjugacy are not classifiable by countable structures.
\end{restatable}

It is quite straightforward to show that if $M$ is a \II factor with a Cartan subalgebra $A$ and $N$ is any \II factor, then classifying the Cartan subalgebras of $M\otb N$ is at least as complicated as classifying the Cartan subalgebras of $N$ (by considering Cartan subalgebras of the form $A\otb B$, with $B\sub N$ Cartan, see Proposition~\ref{prop:MotbN} for the precise statement). Together with Theorem~\ref{thm:R} this then immediately implies the following result on McDuff \II factors.

\begin{cor}\label{cor:McDuff}
	Let $M$ be a McDuff \II factor with at least one Cartan subalgebra. Then the Cartan subalgebras of $M$ up to unitary conjugacy are not classifiable by countable structures.
\end{cor}

\subsection*{Idea behind the proofs}

Let us sketch the proofs of the main structural result (Theorem~\ref{thm:structure}) and of our non-classifiability result for homomorphisms to $E_0$ (Theorem~\ref{thm:nonstrongE_01}), and then indicate briefly how the non-classifiability results for Cartan subalgebras (Theorems~\ref{thm:nonclass} and \ref{thm:nonclassaut}) follow from these.

\vspace{5pt}

\textit{Theorem~\ref{thm:structure}.} Recall our setup for Theorem~\ref{thm:structure}, where $\Gam\in\cC_{rss}$, $N$ is a \II factor, $(X,\mu)$ a standard probability space and $\Gam\act X$ a free ergodic pmp action. We consider $\cM:=(\Linf\cross\Gam)\otb N$. The proof of Theorem~\ref{thm:structure} can be subdivided into the following two steps.

\begin{enumerate}
	\item \textit{Describe the Cartan subalgebras of $\cM\cong (\LinfN)\cross \Gam$ contained in $\LinfN$.}\\
	Note that by step~2 it indeed suffices to characterize these Cartan subalgebras. Also note that we don't need the condition $\Gam\in\cC_{rss}$ until step~2, before which all results hold for any countable group $\Gam$.
	
	Writing $\LinfN=\dint_X N\,d\mu(x)$, we will see in Lemmas~\ref{lem:maxab} and \ref{lem:Ax} that it is quite straightforward from the basic properties of direct integrals that a Cartan subalgebra $A$ contained in $\LinfN$ can be written as a direct integral
	\begin{equation}\label{eq:Aasint}
	A=\dint_X A_x\,d\mu(x)
	\end{equation}
	where each $A_x$ is a Cartan subalgebra of $N$. Moreover, one observes that the canonical unitaries $u_g$ associated to the group elements of $\Gam$ act on $\LinfN=\dint_X N\,d\mu(x)$ by shifting the integral components. Involving Popa's intertwining technique, this will imply the following result of independent interest, see also Lemma~\ref{lem:Cartaniff}.
	\begin{lemma}
		A subalgebra $A$ of $\LinfN$ of the form \eqref{eq:Aasint} is a Cartan subalgebra of $(\Linf\cross\Gam)\otb N$ if and only if the integral components of $A$ within the same $\Gam$-orbit are unitarily conjugate.
	\end{lemma}
	This in particular generalizes a result from Feldman and Moore (\cite[Theorem~II.10]{FM77}) saying that two Cartan subalgebras $A_1,A_2$ of a \II factor $N$ are unitarily conjugate if and only if $\diag(A_1,A_2)$ is a Cartan subalgebra of $M_2(N)$. Indeed, one can view $M_2(N)$ as $(L^\infty(\{0,1\})\cross \Z/2\Z)\otb N$.
	\item \textit{Show that every Cartan subalgebra of $\cM$ is unitarily conjugate to one contained in $\LinfN$.}\\
	This is were the assumption $\Gam\in\cC_{rss}$ comes into play. Using the dichotomy~\ref{def:rss} for such groups, it follows easily that, given any Cartan subalgebra $A\sub\cM$, we must have $A\emb_\cM \LinfN$ (see Theorem~\ref{thm:Popa}). Carefully exploiting the structure of $\cM$, we will see in Lemma~\ref{lem:AuB} that this implies the existence of a Cartan subalgebra $B$ of $\cM$ contained in $\LinfN$ such that $A\emb_\cM B$, which by \cite[Theorem~A.1]{Po01} will allow us to finish the proof of Theorem~\ref{thm:structure}.	
\end{enumerate}

\textit{Theorem~\ref{thm:nonstrongE_01}.} For the proof of Theorem~\ref{thm:nonstrongE_01} we start with any countable group $\Gam$, a standard probability space $(X,\mu)$ and any free ergodic pmp action $\Gam\act X$ that is not strongly ergodic. In order to construct homomorphisms from $\cR(\Gam\act X)$ to $E_0$, the non-strong ergodicity of $\Gam\act X$ will come into play through the use of almost invariant sequences (see Section~\ref{sec:aiseq}). Using the fundamental results of \cite{Dye59,JS87} we will construct a lot of nontrivial almost invariant sequences for the action $\Gam\act X$, namely one for every element ${\bf t}\in (0,1)^\N$. Following \cite{JS87} we can associate a map $X\rarrow \{0,1\}^\N$ to every almost invariant sequence $(A_n)_n$ via
\[
X\rarrow \{0,1\}^\N: x\mapsto (\een_{A_n}(x))_n.
\]
This construction will associate to every element of $(0,1)^\N$ a homomorphism from $\cR(\Gam\act X)$ to $E_0$. Theorem~\ref{thm:nonstrongE_01} will now follow from Hjorth's theory of turbulence (see section~\ref{sec:turb} and \cite{Hjo00}). More specifically we will see in Proposition~\ref{prop:crit} that the existence of a \textquotedblleft nontrivial" homomorphism from $(0,1)^\N$ up to $\ell^1$-equivalence to an equivalence relation $E$ implies that $E$ is not classifiable by countable structures. It turns out that this applies to our construction above, allowing us to finish the proof of Theorem~\ref{thm:nonstrongE_01}.

\vspace{5pt}

\textit{Theorem~\ref{thm:nonclass}.} The proof of Theorem~\ref{thm:nonclass} will rely on the two aforementioned results in the following way.
\begin{itemize}
	\item \textit{Use Theorem~\ref{thm:structure} to reduce the problem to studying $\Hom(\cR(\Gam\act X),\cR(\cU(N)\act\Cartan(N)))$.}\\
	We will see in Lemma~\ref{lem:eqrel} that the proof of Theorem~\ref{thm:structure} implies that for getting the desired non-classifiability result on the equivalence relation of unitary conjugacy on $\Cartan(\cM)$ it is \textquotedblleft enough" to study homomorphisms between $\cR(\Gam\act X)$ and $\cR(\cU(N)\act\Cartan(N))$ up to unitary conjugacy inside $N$ almost everywhere. 
	\item \textit{Pick a \II factor $N$ which already has \textquotedblleft a lot" of Cartan subalgebras and use Theorem~\ref{thm:nonstrongE_01} to get enough such homomorphisms from them.}\\
	The standing assumption in the statement of Theorem~\ref{thm:nonclass} is that $\cR(\cU(N)\act\Cartan(N))$ is not smooth, i.e. we cannot assign real numbers as complete invariants for this equivalence relation. It is known (see \cite[Theorem~1.1]{HKL90}) that this is equivalent to having a Borel reduction of $E_0$ to $\cR(\cU(N)\act\Cartan(N))$. Intuitively this means we can find a copy of $E_0$ inside $\cR(\cU(N)\act\Cartan(N))$. This will allow us to transfer the result in Theorem~\ref{thm:nonstrongE_01} to homomorphisms between $\cR(\Gam\act X)$ and $\cR(\cU(N)\act\Cartan(N))$, finishing the proof of Theorem~\ref{thm:nonclass}.
\end{itemize}

\textit{Theorem~\ref{thm:nonclassaut}.} The proof of Theorem~\ref{thm:nonclassaut} is similar to the proof of Theorem~\ref{thm:nonclass}, albeit slightly more technical. One of the differences is the first bullet point above. The reduction there works for Theorem~\ref{thm:nonclass} because of the fact that two Cartan subalgebras of $\cM$ contained in $\LinfN$ are unitarily conjugate in $\cM$ if and only if the \textquotedblleft slices" from their direct integral decompositions are unitarily conjugate in $N$. This does not hold for conjugation by an automorphism. Nevertheless it turns out that for specific choices of $N$, up to applying a partial automorphism of $X$, $A$ being conjugate by an automorphism to $B$ does imply that their slices are conjugate by a stable automorphism of $N$ on a positive measure subset of $X$. Using this together with the \II factor $N$ from \cite{SV11} and an argument similar to that of the proof of Theorem~\ref{thm:nonstrongE_01} will give us the desired result.

\vspace{10pt}

\textsc{Acknowledgments.} I am very grateful to my advisor Adrian Ioana for suggesting this problem and for the many helpful meetings offering me advice throughout the project. I would also like to thank Stefaan Vaes, R\'emi Boutonnet, and the referee for their valuable comments which helped improve the exposition of the paper and generalize some of its results.

\section{Preliminaries}

In order to make this article as self-contained as possible, we include a rather extensive section of preliminaries, including the necessary ingredients from both von Neumann algebras and descriptive set theory.

\subsection{von Neumann algebras}

In this paper we consider separable tracial von Neumann algebras $(M,\tau)$, i.e. von Neumann algebras with a faithful normal tracial state $\tau:M\rarrow\C$. We denote by $\cU(M)$ the unitary group of $M$, by $\cZ(M):=M'\cap M$ the center of $M$, and by $L^2(M)$ the completion of $M$ under the Hilbert norm $\norm{x}_2=\tau(x^*x)^{\frac{1}{2}}$. We note the well-known fact that this 2-norm turns $\cU(M)$ into a Polish group. For a projection $p\in M$, we denote by $z(p)$ its central support, i.e. the smallest projection $z\in\cZ(M)$ such that $zp=p$. For a von Neumann subalgebra $P\sub M$ we denote by $E_P:M\rarrow P$ the conditional expectation from $M$ onto $P$, by $e_P:L^2(M)\rarrow L^2(P)$ the orthogonal projection onto $L^2(P)$ and by $\cN_M(P):= \{u\in\cU(M)\mid uPu^*=P\}$ the normalizer of $P$ in $M$. \emph{Jones' basic construction} of the inclusion $P\sub M$ is the von Neumann subalgebra of $B(L^2(M))$ generated by $M$ and $e_P$, and is denoted by $\ip{M}{e_P}$. We will write $\Aut(M)$ for the group of (trace preserving) automorphisms of $(M,\tau)$. On $\Aut(M)$ we put the topology for which a net $(\al_i)$ converges to $\al$ if and only if for all $x\in M$, we have $\lim_i\norm{\al_i(x)-\al(x)}_2= 0$. One can show that this turns $\Aut(M)$ into a Polish group.

Usually we will work with \II factors, i.e. infinite dimensional tracial von Neumann algebras with trivial center $\cZ(M)=\C 1$. A \emph{Cartan subalgebra} of a \II factor $M$ is a maximal abelian von Neumann subalgebra that is \emph{regular}, i.e. $\cN_M(A)''=M$. Given a \II factor $M$, we write $\Cartan(M)$ for the space of Cartan subalgebras of $M$. We say that two Cartan subalgebras $A$ and $B$ are \emph{unitarily conjugate in $M$} if there exists $u\in \cU(M)$ such that $uAu^*=B$, we say they are \emph{conjugate by an automorphism of $M$} if there exists $\al\in\Aut(M)$ such that $\al(A)=B$ and we say they are \emph{conjugate by a stable automorphism of $M$} if there exist nonzero projections $p\in A$, $q\in B$ and a $*$-isomorphism $\al:pMp\rarrow qMq$ such that $\al(Ap)=Bq$. We will write $A\sim_u B$, $A\sim_a B$, and $A\sim_{sa} B$ respectively for these notions of equivalence.

\subsubsection{Intertwining-by-bimodules}

In \cite{Po03} Popa proved the following powerful technique for conjugating (corners of) subalgebras of a tracial von Neumann algebra.

\begin{theorem}[{\cite[Theorem~2.1 and Corollary~2.3]{Po03}}]\label{thm:Popa}
	Let $(M,\tau)$ be a tracial von Neumann algebra and $P\sub pMp$, $Q\sub qMq$ be von Neumann subalgebras. Let $\cU\sub\cU(P)$ be a subgroup such that $\cU''=P$. Then the following are equivalent.
	\begin{itemize}
		\item There exist projections $p_0\in P$, $q_0\in Q$, a $^*$-homomorphism $\psi:p_0Pp_0\rarrow q_0Qq_0$ and a nonzero partial isometry $v\in q_0Mp_0$ such that $\psi(x)v=vx$ for all $x\in p_0Pp_0$.
		\item There is no sequence $(u_n)_n\in\cU$ such that $\norm{E_Q(x^*u_ny)}_2\rarrow 0$ for all $x,y\in pMq$.
	\end{itemize}
\end{theorem}

\begin{terminology}
	If one of the equivalent conditions of the above theorem holds, we say that \emph{a corner of} $P$ \emph{embeds into} $Q$ \emph{inside} $M$, and we write $P\emb_M Q$. If $Pp'\emb_M Q$ for any nonzero projection $p'\in P'\cap pMp$, then we say that $P$ embeds strongly into $Q$ inside $M$, and we write $P\emb_M^s Q$.
\end{terminology}

\subsubsection{Relative amenability}

Recall that a tracial von Neumann algebra $(M,\tau)$ is called \textit{amenable} if there exists a positive linear functional $\vphi:B(L^2(M))\rarrow\C$ such that $\vphi|_M=\tau$ and $\vphi$ is $M$-central, i.e. $\vphi(xT)=\vphi(Tx)$ for all $x\in M$, $T\in B(L^2(M))$. In \cite{OP07} Ozawa and Popa introduced the analogous notion of relative amenability:

\begin{definition}
	Suppose $P\sub pMp$, $Q\sub M$ are von Neumann subalgebras. Then we say that $P$ is \emph{amenable relative to} $Q$ \emph{inside} $M$ if there exists a positive linear functional $\vphi:p\ip{M}{e_Q}p\rarrow \C$ such that $\vphi|_{pMp}=\tau$ and $\vphi$ is $P$-central.
\end{definition}

\subsubsection{Relatively strongly solid groups}

Recall that a von Neumann algebra is called \emph{diffuse} if it does not contain a minimal projection. In \cite{Oz03} Ozawa established that the group von Neumann algebra of a non-elementary hyperbolic group is \emph{solid}: the relative commutant of any diffuse von Neumann subalgebra is amenable. Later Ozawa and Popa strengthened this result in \cite{OP08} for free groups by proving that $L\bF_n$, $2\leq n\leq \infty$, is \emph{strongly solid}: the normalizer of any diffuse amenable von Neumann subalgebra is still amenable. Later Chifan and Sinclair showed in \cite{CS11} that this actually holds for all non-elementary hyperbolic groups.

The next breakthrough was realized by Popa and Vaes who showed in \cite{PV11,PV12} that non-abelian free groups and more generally non-elementary hyperbolic groups are \emph{relatively strongly solid}. Here we use the terminology from \cite[Definition~2.7]{CIK13}.

\begin{definition}\label{def:rss}
	A countable non-amenable group $\Gam$ is called \emph{relatively strongly solid} if for any tracial crossed product $M:=B\cross\Gam$ and all von Neumann subalgebras $A\sub M$ with $A$ amenable relative to $B$ inside $M$ we have either
	\begin{enumerate}
		\item $A\emb_M B$, or
		\item $\cN_M(A)''$ is still amenable relative to $B$ inside $M$.
	\end{enumerate}
	We denote the class of relatively strongly solid groups by $\cC_{rss}$.
\end{definition}

Note that it follows from Definition~\ref{def:rss} that $L^\infty(X)\cross\Gamma$ has $L^\infty(X)$ as its unique Cartan subalgebra up to unitary conjugacy if $\Gam\in\cC_{rss}$ and the action is free and ergodic, see also \cite{PV11, PV12}.

\subsubsection{Direct integrals}\label{sec:directint}

In this section we recall the basic definitions and properties we need from the theory of direct integral decompositions of von Neumann algebras. A lot of it is taken from chapter 14 of \cite{KRII}. We start with the direct integral of Hilbert spaces. Recall that a standard probability space $(X,\mu)$ is a probability space whose (Borel) $\sigma$-algebra is generated by the open sets of some Polish topology on $X$.

\begin{definition}\label{def:dint}
	Let $(X,\mu)$ be a standard probability space. Let $\{\cH_x\}$ be a family of separable Hilbert spaces indexed by the points $x$ of $X$. A separable Hilbert space $\cH$ is the \emphind{direct integral} of $\{\cH_x\}$ over $(X,\mu)$, written $\cH=\dint_X \cH_x\,d\mu(x)$, if for each $\xi\in\cH$ there exists a function $x\mapsto \xi(x)$ on $X$ such that $\xi(x)\in\cH_x$ for every $x$ and
	\begin{enumerate}[label=(\roman*)]
		\item $x\mapsto \ip{\xi(x)}{\eta(x)}$ is integrable and $\ip{\xi}{\eta}=\int_X \ip{\xi(x)}{\eta(x)}\,d\mu(x)$ for all $\xi,\eta\in\cH$,
		\item if $\xi_x\in\cH_x$ for every $x\in X$ and $x\mapsto\ip{\xi_x}{\eta(x)}$ is integrable for every $\eta\in\cH$ then there exists $\xi\in\cH$ such that $\xi(x)=\xi_x$ for almost every $x\in X$.
	\end{enumerate}
	We call $\dint_X \cH_x\,d\mu(x)$ and $x\mapsto \xi(x)$ the \emph{direct integral decompositions} of $\cH$ and $\xi$ respectively.
\end{definition}

Two very easy examples are the following.

\begin{example}
	\begin{enumerate}
		\item The (discrete) direct sum of countably many Hilbert spaces $\{\cH_n\}$ can be viewed as the direct integral of $\{\cH_n\}$ over the natural numbers with the counting measure.
		\item Given $(X,\mu)$ as above it is easy to check that $L^2(X,\mu)=\dint_X \C\,d\mu(x)$.
	\end{enumerate}
\end{example}

Once we have a direct integral of Hilbert spaces, we can consider the appropriate notions of the direct integral of operators and after that of von Neumann algebras as well.

\begin{definition}
	If $\cH=\dint_X \cH_x\,d\mu(x)$, an operator $T\in \cB(\cH)$ is called \emph{decomposable}\index{decomposable operator} if there is a function $x\mapsto T(x)$ on $X$ such that $T(x)\in\cB(\cH_x)$ for every $x$ and such that for every $\xi\in\cH$, $T(x)\xi(x)=(T\xi)(x)$ for almost every $x$. If $T(x)=f(x)I_x$ for almost every $x$, we say that $T$ is \emph{diagonalizable}.
\end{definition}

\begin{remark}
	It is easy to show that for $\xi,\eta\in\cH$, respectively $S,T\in\cB(\cH)$ decomposable, we have $\xi=\eta$ if and only if $\xi(x)=\eta(x)$ almost everywhere, respectively $S=T$ if and only if $S(x)=T(x)$ almost everywhere.
\end{remark}

The following proposition tells us that direct integrals commute with all the basic operations, which allows us to manipulate them easily.

\begin{proposition}[{\cite[Proposition~14.1.8]{KRII}}]\label{prop:basicprop}
	If $T, T_1,T_2$ are decomposable operators on $\cH=\dint_X \cH_x\,d\mu(x)$, then $aT_1+T_2, T_1T_2$ and $T^*$ are decomposable. Moreover the following hold for almost every $x$:
	\begin{enumerate}[label=(\roman*),font=\normalfont]
		\item $(aT_1+T_2)(x)=aT_1(x)+T_2(x)$,
		\item $(T_1T_2)(x)=T_1(x)T_2(x)$,
		\item $T^*(x)=T(x)^*$.
	\end{enumerate}
\end{proposition}

\begin{theorem}[{\cite[Theorem~14.1.10]{KRII}}]
	If $\cH=\dint_X \cH_x\,d\mu(x)$, then the set $\cR$ of decomposable operators is a von Neumann algebra. Moreover $\cR$ has abelian commutant consisting of the algebra $\cC$ of diagonalizable operators.
\end{theorem}

\begin{definition}
	A von Neumann algebra $M$ on $\cH=\dint_X \cH_x\,d\mu(x)$ is called \emph{decomposable}\index{decomposable von Neumann algebra} if it is a subalgebra of the (von Neumann) algebra of decomposable operators.
\end{definition}

\begin{proposition}[{\cite[Proposition~14.1.18]{KRII}}]\label{prop:decvNa}
	If $M$ is a decomposable von Neumann algebra on $\cH=\dint_X \cH_x\,d\mu(x)$ containing the algebra of diagonalizable operators, then there exist von Neumann algebras $M_x$ on $\cB(\cH_x)$ such that $M=\dint_X M_x\,d\mu(x)$ in the following sense: If $T\in \cB(\cH)$ is a decomposable operator, then $T\in M$ if and only if $T(x)\in M_x$ almost everywhere. Moreover, if $N$ is a von Neumann algebra with decomposition $N=\dint_X N_x\,d\mu(x)$ and $M_x=N_x$ almost everywhere, then $M=N$. 
\end{proposition}

\begin{remark}
	In \cite{KRII} a decomposable von Neumann algebra $M$ is defined through the existence of a norm-separable strong operator dense C$^*$-subalgebra $A$ such that the identity representation $\iota$ is decomposable and $\iota_x(A)$ is strong operator dense in $M_x$ almost everywhere. It is then shown (\cite[Theorem~14.1.16]{KRII}) that this is equivalent to our definition above and that the decomposition $x\mapsto M_x$ is unique.
\end{remark}

\begin{proposition}[{\cite[Proposition~14.1.24]{KRII}}]\label{prop:Comm}
	Let $M$ be a decomposable von Neumann algebra on $\cH=\dint_X \cH_x\,d\mu(x)$ containing the algebra $\cC$ of diagonalizable operators, with decomposition $x\mapsto M_x$. Then $M'$ is also decomposable and $(M')_x=(M_x)'$ almost everywhere.
\end{proposition}

The following easy lemma describes the maximal abelian subalgebras of a decomposable von Neumann algebra and will be used several times in section~\ref{sec:structure}. 

\begin{lemma}\label{lem:maxab}
	Let $M$ be a decomposable von Neumann algebra on $\cH=\dint_X \cH_x\,d\mu(x)$ containing the algebra $\cC$ of diagonalizable operators, with decomposition $x\mapsto M_x$. Let $A=\dint_X A_x\,d\mu(x)$ be a von Neumann subalgebra of $M$. Then $A$ is maximal abelian inside $M$ if and only if $A_x$ is maximal abelian inside $M_x$ for almost every $x\in X$.
\end{lemma}
\begin{proof}
	By Proposition~\ref{prop:Comm} it follows that
	\[
	A'\cap M = \left(\dint_X A_x\,d\mu(x)\right)'\cap \dint_X M_x\,d\mu(x) = \dint_X A_x'\cap M_x\,d\mu(x).
	\]
	If $A_x$ is maximal abelian inside $M_x$ for almost every $x\in X$, this implies that
	\[
	A'\cap M = \dint_X A_x'\cap M_x\,d\mu(x) = \dint_X A_x\,d\mu(x) = A.
	\]
	If on the other hand $A$ is maximal abelian inside $M$, we get
	\[
	\dint_X A_x\,d\mu(x) = A = A'\cap M = \dint_X A_x'\cap M_x\,d\mu(x),
	\]
	and it follows from the uniqueness of the decomposition that $A_x=A_x'\cap M_x$ almost everywhere.
\end{proof}

We end with the following theorem which will be helpful for us in section~\ref{sec:Aut}.

\begin{theorem}\label{thm:Tak-Desint}
	Suppose $(M_1,\tau_1)=\dint_X (M_1(x),\tau_1(x))d\mu_1(x)$ and $(M_2,\tau_2)=\dint_Y (M_2(x),\tau_2(x))d\mu_2(x)$ are direct integrals of tracial von Neumann algebras and $\al:M_1\rarrow M_2$ is a trace preserving automorphism such that $\al(L^\infty(X))=L^\infty(Y)$. Then there exist full measure sets $X'\sub X$, $Y'\sub Y$, and a Borel isomorphism $\Phi: Y'\rarrow X'$ with $\Phi(\mu_2)$ equivalent to $\mu_1$ such that $\al$ decomposes into tracial isomorphisms $\{\al_x:M_1(x)\rarrow M_2(\Phi^{-1}(x))\}$.
\end{theorem}
\begin{proof}
	Since $\al$ is trace preserving, it gives rise to a unitary $U:L^2(M_1)\rarrow L^2(M_2)$ which implements $\al$. The result then immediately follows from \cite[Theorem~IV.8.23]{Tak01}.
\end{proof}

\subsection{Complexity of classification}

In the following we will review some of the set-theoretic notions allowing us to talk about the exact complexity of a classification problem. A good reference is \cite{KT-D13}, where several of the definitions below are taken from.

\subsubsection{First definitions and results}

Given an equivalence relation $E$ on a space $X$, a \emph{(complete) classification} of $X$ up to $E$ consists of a set of invariants $I$ and a map $f:X\rarrow I$ such that $xEy\iff f(x)=f(y)$. Most often the base space $X$ is a standard Borel space, i.e. a Polish space with the Borel $\sigma$-algebra generated by the open sets, and the equivalence relation $E$ is Borel or analytic (as a subspace of $X\times X$). The following important notion gives us a way of comparing equivalence relations.

\begin{definition}\label{def:BorelReduction}
	Let $(X,E)$ and $(Y,F)$ be equivalence relations on standard Borel spaces. Then $E$ is \emph{(Borel) reducible} to $F$, written $E\leq_B F$, if there is a Borel map $f:X\rarrow Y$ such that $xEy\iff f(x)Ff(y)$.
\end{definition}

This is exactly saying that the $F$-classes are complete invariants for the $E$-classes and intuitively means that the classification problem for $E$ is at most as complicated as the one for $F$. If both $E\leq_B F$ and $F\leq_B E$ we say that $E$ is \emph{(Borel) bi-reducible} to $F$ and write $E\sim_B F$. If $E\leq_B F$ but $F\not\leq_B E$, we write $E<_B F$.

For any Polish space $Y$ we can consider the equality relation $=_Y$ on $Y$ given by $\{(x,y)\in Y^2\mid x=y\}$. Denoting by $n$, for $n\in\N$, any set of cardinality $n$ we then have $E\sim_B (=_n)$ for any Borel equivalence relation $E$ with $n$ equivalence classes. We thus get (dropping $=$ for clarity) that
\[
1 <_B 2 <_B 3 <_B \dots <_B \N
\]
is an initial segment of $\leq_B$, as $\N\leq_B E$ for any Borel equivalence relation $E$ with infinitely many equivalence classes. The following dichotomy theorem of Silver extends this result and tells us that $\R\leq_B E$ if $E$ has uncountably many equivalence classes.

\begin{theorem}[{\cite{Sil80}}]
	Let $E$ be a Borel equivalence relation. Then either $E\leq_B \N$ or $\R\leq_B E$.
\end{theorem}

\begin{definition}\label{def:smooth}
	An equivalence relation $E$ on $X$ is \emph{smooth} (or \emph{concretely classifiable}) if $E\leq_B (=_Y)$ for some Polish space $Y$, i.e. there is a Borel map $f:X\rarrow Y$ such that $xEy\iff f(x)=f(y)$.
\end{definition}

Note that this is equivalent to asking that $E\leq_B (=_\R)$, since all Polish spaces are Borel isomorphic. In a way smooth equivalence relations can still be considered \textquotedblleft easy". (Un)fortunately, not all equivalence relations are smooth.

\begin{example}
	A very important non-smooth equivalence relation is the relation $E_0$ on $2^\N$ where 
	\[
	{\bf x}E_0{\bf y} \quad \Leftrightarrow \quad \exists N\in\N, \forall n\geq N: x_n=y_n.
	\]
	To see that $E_0$ is not smooth, assume that we have a Borel reduction $f:2^\N\rarrow [0,1]$ from $E_0$ to $(=_{[0,1]})$. Let $\mu$ be the usual product measure on $2^\N$. Then $f^{-1}([0,\frac{1}{2}])$ and $f^{-1}([\frac{1}{2},1])$ are both tail events, so applying Kolmogorov's zero-one law, we get that either $\mu(f^{-1}([0,\frac{1}{2}]))=1$ or $\mu(f^{-1}([\frac{1}{2},1]))=1$. Continuing cutting intervals in half we get in this way that $f$ is $\mu$-almost everywhere constant, which contradicts the fact that it is a Borel reduction.
\end{example}

The following result, usually referred to as the Glimm-Effros dichotomy, shows that $E_0$ is the \textquotedblleft smallest" among non-smooth Borel equivalence relations.

\begin{theorem}[\cite{HKL90}]\label{thm:E_0smallest}
	If $E$ is a Borel equivalence relation on a Polish space $X$ and $E$ is not smooth, then $E_0\leq_B E$. Moreover one can find a continuous injective Borel reduction $f:2^\N\rarrow X$.
\end{theorem}

It follows from the above discussion that
\[
1 <_B 2 <_B 3 <_B \dots <_B \N <_B \R <_B E_0
\]
forms an initial segment for $\leq_B$ and moreover $E_0\leq E$ for any non-smooth Borel equivalence relation $E$. Beyond $E_0$ the situation becomes way more complicated and one will for instance not find a unique minimal equivalence relation above $E_0$ (see \cite[Theorem~2]{KL97}). However, one natural thing to consider for general equivalence relations is \emph{classification by countable structures}.

\subsubsection{Classification by countable structures}\label{sec:classctbstruct}

Intuitively, an equivalence relation is classifiable by countable structures if we can assign complete invariants built out of countable (or \textquotedblleft discrete") structures of some given type, e.g. groups, graphs, fields, etc. More concretely, we have the following.

\begin{definition}
	A \textit{countable signature} is a countable family $L=\{f_i\}_{i\in I} \cup \{R_j\}_{j\in J}$ of function symbols $f_i$ and relation symbols $R_j$. We denote by $n_i\geq 0$, resp. $m_j\geq 1$, the arity of $f_i$, resp. $R_j$.
	
	An $L$-\emph{structure} is given by
	\[
	\cA:=(A,\{f_i^\cA\}_{i\in I}, \{R_j^\cA\}_{j\in J}),
	\]
	where $A$ is a nonempty set, $f_i^\cA:A^{n_i}\rarrow A$ are functions, and $R_j^\cA\sub A^{m_j}$ are relations.
\end{definition}

\begin{example}
	Consider $L=\{\cdot,1\}$, where $\cdot$ is a binary and $1$ is a nullary function symbol. Then a group is any $L$-structure $(G,\cdot^G,1^G)$ satisfying the group axioms. Similarly, using other signatures, one can describe rings, fields, graphs, etc.
\end{example}

When considering only countable structures, we can of course always take $A=\N$ up to isomorphism. This gives rise to the following.

\begin{definition}
	Given a countable signature $L$, we consider the \emph{space of countable $L$-structures}
	\[
	X_L:= \prod_{i\in I}\N^{(\N^{n_i})} \times \prod_{j\in J} 2^{(\N^{m_j})}.
	\]
	Putting the discrete topologies on $\N$ and $2$, $X_L$ becomes a Polish space for the product topology.
\end{definition}

Now let $S_\infty$ be the group of \textit{all} permutations of $\N$. Then we have an obvious action $S_\infty\act X_L$, called the \emph{logic action}. One easily checks that this action induces the equivalence relation of isomorphism in $X_L$, i.e. for $\cA,\cB\in X_L$ we have $\cA\cong\cB\,\Leftrightarrow\,\exists g\in S_\infty: g\cdot \cA=\cB$. We will denote by $\cong_L$ the equivalence relation of isomorphism on $X_L$.

\begin{definition}
	Let $E$ be an equivalence relation on a standard Borel space $X$. Then we say that $E$ is \emph{classifiable by countable structures} if there exists a countable signature $L$ such that $E\leq_B (\cong_L)$.
\end{definition}

\begin{example}
	\begin{itemize}
		\item If $E$ is smooth, then $E$ admits classification by countable structures.
		\item (\cite{Kec92}) If $G$ is a Polish locally compact group and $X$ is a Borel $G$-space, then the orbit equivalence relation of $G\act X$ admits classification by countable structures.
	\end{itemize}
\end{example}

\subsubsection{Turbulence and generic ergodicity}\label{sec:turb}

The basic method for showing that some equivalence relation is not classifiable by countable structures was developed by Hjorth (see \cite{Hjo00}) and is called \emph{turbulence}. In the following, let $G$ be a Polish group acting continuously on a Polish space $X$.

\begin{definition}
	Let $U\sub X$ be a nonempty open set and $V\sub G$ be an open neighborhood of $1\in G$. For $x\in U$ we define
	\begin{align*}
	\cO(x,U,V):= \{y\in U\mid &\exists k\in\N, x_0,x_1,\dots,x_k\in U, g_0,\dots,g_{k-1}\in V\\ &\text{ such that } x=x_0, y=x_k \text{ and }\forall i<k:x_{i+1}=g_ix_i\}.
	\end{align*}
	We call $\cO(x,U,V)$ the \emph{local} $U,V$-\emph{orbit of} $x$.
\end{definition}

Recall that a subspace of a Polish space is called \emph{meager} if it is disjoint from a dense $G_\del$ set and \emph{comeager} if it includes one. We can think of these as being \textquotedblleft small", respectively \textquotedblleft large", sets.

\begin{definition}
	The action $G\act X$ is \emph{turbulent} if
	\begin{enumerate}[label=(\roman*)]
		\item every orbit is dense,
		\item every orbit is meager,
		\item for every $x\in X$, for every $U\sub X$ open and every $V\sub G$ open with $x\in U$, $1\in V$: $\overline{\cO(x,U,V)}$ has nonempty interior.
	\end{enumerate}
\end{definition}

Note that these conditions give rise to \textquotedblleft turbulence" on different scales. Orbits being dense is a global phenomenon, telling us we can get to every region of our space by applying group elements to any chosen point in our space. Moreover, one can easily show (see Example~\ref{eg:denseorbits} below) that an action satisfying (i) and (ii) from the above definition is already \textquotedblleft turbulent enough", in the sense that it cannot be smooth. The main condition in the above definition is (iii) though, which says that even on a small scale, both in the group and in the space, the action exhibits turbulent behaviour.

\begin{example}\label{eg:l1}
	Consider $(\ell^1,+)$ with the usual $1$-norm ($\norm{g}_1=\sum_i\abs{g_i}$) acting by translation on $(\R^\N,+)$ with the product topology. It is not hard to see that this action is turbulent (see also \cite[Proposition~3.25]{Hjo00}). Every orbit is dense and meager because $\ell^1$ is dense and meager. To get the third condition, fix $x\in \R^\N$, $U\sub\R^\N$ open containing $x$ and $V\sub\ell^1$ an open neighborhood of the identity. By shrinking $U$ and $V$ if necessary, we can assume that for some $\eps>0, l\in\N$ we have
	\begin{align*}
	U &= \{z\in \R^\N \mid \forall i<l: \abs{z_i-x_i}<\eps\},\\
	V &= \{z\in \ell^1 \mid \norm{z}_1<\eps\}.
	\end{align*}
	Now let $U_0\sub U$ be any open set, then it suffices to find some $z\in \cO(x,U,V)\cap U_0$. Using the density of the orbit of $x$, we can find $g\in\ell^1$ with $g\cdot x\in U_0$. Choose now $k\in\N$ large enough so that $\norm{g}_1/k < \eps$ and put $h=g/k$, i.e. $h_i=g_i/k$ for every $i\in\N$. Then $h\in V$ and we can put $x_0:=x, x_{j+1} := h\cdot x_j$. By convexity of $U$, we have that $x_j\in U$ for every $0\leq j\leq k$.	Hence $x_k=g\cdot x\in \cO(x,U,V)\cap U_0$ and the result follows.
\end{example}

Next we consider what is called \emph{generic ergodicity}. For this we first need the obvious notion of a homomorphism between equivalence relations.

\begin{definition}\label{def:hom}
	Let $E$ and $F$ be equivalence relations on Polish spaces $X$ and $Y$. A \emph{homomorphism} from $E$ to $F$ is a Borel function $f:X\rarrow Y$ such that $xEy\Rightarrow f(x)Ff(y)$.
\end{definition}

Given $E$ and $F$ equivalence relations on Polish spaces $X$ and $Y$, we will write $\Hom(E,F)$ for the space of homomorphisms from $E$ to $F$.

\begin{definition}
	For $E,F$ as above, we say that $E$ is \emph{generically $F$-ergodic} if for every homomorphism $f$ between $E$ and $F$, there is a comeager set $A\sub X$ such that $f$ maps $A$ to a single $F$-class.
\end{definition}

\begin{example}\label{eg:denseorbits}
	If $E$ denotes the orbit equivalence relation for some continuous action of a Polish group $G$ on a Polish space $X$ with a dense orbit, then $E$ is easily seen to be generically $(=_\R)$-ergodic, see for instance \cite[Example~2.17]{KT-D13}. In particular, if every orbit is also meager, then $E$ is not smooth.
\end{example}

The following important result of Hjorth tells us that turbulence is an obstruction for classification by countable structures.

\begin{theorem}[\cite{Hjo00}]\label{thm:Hjo00}
	If a Polish group $G$ acts turbulently on a Polish space $X$, then $\cR(G\act X)$ is generically $\cong_L$-ergodic for any countable signature $L$. In particular, since turbulent actions have meager orbits by definition, $\cR(G\act X)$ does not admit classification by countable structures.
\end{theorem}

\subsubsection{A non-classifiability criterion}

The following criterion follows easily from Example~\ref{eg:l1} and Theorem~\ref{thm:Hjo00}.

\begin{proposition}\label{prop:crit}
	If $E$ is an equivalence relation on a Polish space $X$ such that there exists a Borel map $f:(0,1)^\N\rarrow X$ satisfying
	\begin{enumerate}
		\item $f({\bf x})Ef({\bf y})$ whenever ${\bf x}-{\bf y}\in \ell^1$, and
		\item there is no comeager set $A$ in $(0,1)^\N$ which is mapped into a single $E$-class by $f$,
	\end{enumerate}
	then $E$ is not classifiable by countable structures.
\end{proposition}
\begin{proof}
	(This proof follows the same lines of reasoning as \cite[Criterion 3.3]{Lup13}.)
	\begin{Claim}
		Suppose $F,G,R$ are equivalence relations on Polish spaces $Y_1,Y_2,Z$ respectively such that $G$ is generically $R$-ergodic. If there exists a homomorphism $g:Y_2\rarrow Y_1$ from $G$ to $F$ such that $g(C)$ is comeager in $Y_1$ for all comeager $C\sub Y_2$, then $F$ is generically $R$-ergodic.
	\end{Claim}
	\begin{proofofclaim}
		Suppose $h:Y_1\rarrow Z$ is a homomorphism from $F$ to $R$. Then $h\circ g:Y_2\rarrow Z$ is a homomorphism from $G$ to $R$. Since $G$ is generically $R$-ergodic, there exists a comeager set $C\sub Y_2$ that is mapped onto one $R$-class. Hence $g(C)$ is a comeager subset of $Y_1$ that is mapped onto a single $R$-class by $h$.
	\end{proofofclaim}

	\begin{Claim}
		Suppose $E$ and $F$ are equivalence relations on $X$ and $Y$ respectively, and $F$ is generically $\cong_L$-ergodic for every countable signature $L$. If there is a homomorphism $f:Y\rarrow X$ from $F$ to $E$ such that there is no comeager set in $Y$ which is mapped onto a single $E$-class, then $E$ is not classifiable by countable structures.
	\end{Claim}
	\begin{proofofclaim}
		Let $L$ be a countable signature and assume that we have a Borel reduction $g:X\rarrow X_L$ from $E$ to $\cong_L$. Then $g\circ f:Y\rarrow X_L$ is a homomorphism from $F$ to $\cong_L$ and since $F$ is generically $\cong_L$-ergodic, it follows that there is a comeager set $C\sub Y$ which is mapped to a single $\cong_L$-class. Since $g$ is a Borel reduction, this implies that $f(x)Ef(y)$ for all $x,y\in C$, contradicting the fact that no $E$-class has a comeager preimage.
	\end{proofofclaim}
	
	\vspace{5pt}
	
	Let now $E$ be as in the proposition, $G$ the relation of equivalence modulo $\ell^1$ on $\R^\N$, and $F$ the relation of equivalence modulo $\ell^1$ on $(0,1)^\N$. Consider the functions
	\begin{align*}
	g'&:\R^\N\rarrow (-1,1)^\N: (t_n)_n\mapsto \left(\frac{t_n}{\abs{t_n}+1}\right)_n,\\
	\varphi&:(-1,1)\rarrow (0,1): x\mapsto \frac{1}{2}x+\frac{1}{2},\\
	g&:=\varphi^\N\circ g':\R^\N\rarrow (0,1)^\N.
	\end{align*}
	It is easy to see that $g$ satisfies the conditions of Claim~1. Since $G$ is generically $\cong_L$-ergodic for any countable signature $L$ by Example~\ref{eg:l1} and Theorem~\ref{thm:Hjo00}, the same holds for $F$. The result then follows from Claim~2.
\end{proof}

\subsection{The standard Borel space of von Neumann algebras}

In \cite{Eff65} Effros showed that there is a standard Borel structure on the space of von Neumann algebras $\vNa(\cH)$ on a given separable Hilbert space $\cH$. Moreover it follows from his results that the set of von Neumann subalgebras $\vNa(M)$ of a given separable \II factor $(M,\tau)$ is a standard Borel space. In this case one can check that its standard Borel structure is given by the smallest $\sigma$-algebra such that
\[
A\mapsto \tau(E_A(x)y)
\]
is measurable for all $x,y\in M$. Speelman and Vaes then showed the following for the space of Cartan subalgebras $\Cartan(M):=\{A\sub M\mid A \text{ is a Cartan subalgebra of } M\}$.

\begin{proposition}[{\cite[Proposition~12]{SV11}}]\label{prop:SV-Cartan}
	In the above setting the following hold.
	\begin{itemize}
		\item $\Cartan(M)\sub\vNa(M)$ is a Borel set and hence a standard Borel space.
		\item The equivalence relation of unitary conjugacy on $\Cartan(M)$ is Borel (i.e. as a subset of $\Cartan(M)\times\Cartan(M)$).
		\item The equivalence relation of conjugacy by an automorphism on $\Cartan(M)$ is analytic.
	\end{itemize}
\end{proposition}

\begin{remark}
	In \cite[Theorem~2]{SV11} the authors construct a \II factor for which the equivalence relation of conjugacy by an automorphism on Cartan subalgebras is completely analytic and hence not Borel.
\end{remark}

\subsection{Almost invariant sequences}\label{sec:aiseq}

We briefly recall the notions and some of the results from \cite{JS87} because we will use them crucially in section~\ref{sec:nonclass}. Let $\Gam\act (X,\mu)$ be an ergodic pmp action of a countable group $\Gam$ on a probability space $(X,\mu)$. A sequence $(A_n)_{n\geq 1}$ is called \textit{almost invariant} (or asymptotically invariant) if
\[
\lim_n \mu(A_n\Del gA_n) = 0
\]
for every $g\in\Gam$. An almost invariant sequence is \textit{trivial} if $\lim_n \mu(A_n)(1-\mu(A_n))=0$. The action $\Gam\act X$ is called \textit{strongly ergodic} if every almost invariant sequence is trivial. Given two actions of countable groups on probability spaces $\Gam\act (X,\mu)$ and $\Lam\act (Y,\nu)$, we say that a measurable function $\theta:X\rarrow Y$ is a \textit{factor map} if it is measure preserving (so in particular essentially onto) and if $\theta(\Gam x)=\Lam\theta(x)$ for almost every $x\in X$. The above notions are linked in the following way.

\begin{theorem}[{\cite[Theorem~2.1]{JS87}}]\label{thm:JS}
	Let $\Gam\act (X,\mu)$ be an ergodic pmp action of a countable group $\Gam$ on a probability space $(X,\mu)$. Then the following are equivalent.
	\begin{enumerate}
		\item The action $\Gam\act X$ is not strongly ergodic,
		\item There exists an ergodic pmp action of $\Z$ on a probability space $(Y,\nu)$ and a factor map $\theta:X\rarrow Y$ for the actions of $\Gam$ and $\Z$ respectively.
	\end{enumerate}
\end{theorem}

\begin{lemma}\label{lem:aiseq}
	Let $\Gam\act (X,\mu)$ and $\Lam\act (Y,\nu)$ be ergodic pmp actions of countable groups on probability spaces. Suppose $\theta:X\rarrow Y$ is a factor map and $(B_n)_n$ is an almost invariant sequence for $\Lam\act (Y,\nu)$. Put $A_n=\theta\inv(B_n)$. Then $(A_n)_n$ is an almost invariant sequence for $\Gam\act (X,\mu)$. Moreover, if $\lim_n \nu(\cup_{k\geq n} sB_k\Del B_k)=0$ for every $s\in\Lam$, then also $\lim_n \mu(\cup_{k\geq n} gA_k\Del A_k)=0$ for every $g\in\Gam$.
\end{lemma}
\begin{proof}
	We will prove the moreover part, the fact that $(A_n)_n$ is an almost invariant sequence can be done in exactly the same way by dropping \textquotedblleft $\cup_{k\geq n}$" everywhere. Fix $g\in\Gam$ and let $\eps>0$. By the assumptions we can take $S\sub \Lam$ finite such that
	\[
	\mu(\{x\in X\mid \theta(g\inv x)\in S\theta(x) \})\geq 1-\eps.
	\]
	Writing $X_0:=\{x\in X\mid \theta(g\inv x)\in S\theta(x) \}$ we then get
	\begin{align*}
	\mu(\cup_{k\geq n} gA_k\setminus A_k) &= \mu(\cup_{k\geq n} g\theta\inv(B_k)\setminus \theta\inv(B_k))\\
	&= \mu(\{x\in X\mid \exists k\geq n:\theta(g\inv x)\in B_k, \theta(x)\notin B_k\})\\
	&\leq \eps + \mu(X_0\cap \{x\in X\mid \exists k\geq n:\theta(g\inv x)\in B_k, \theta(x)\notin B_k\})\\
	&= \eps + \mu(\cup_{s\in S}\{x\in X_0\mid \theta(g\inv x)=s\theta(x) \text{ and } \exists k\geq n:s\theta(x)\in B_k, \theta(x)\notin B_k\})\\
	&\leq \eps + \sum_{s\in S}\mu(\{x\in X_0\mid \exists k\geq n:s\theta(x)\in B_k, \theta(x)\notin B_k\})\\
	&\leq \eps + \sum_{s\in S}\mu(\cup_{k\geq n}\theta\inv(s\inv B_k)\setminus \theta\inv(B_k))\\
	&= \eps + \sum_{s\in S}\nu(\cup_{k\geq n} s\inv B_k\setminus B_k),
	\end{align*}
	which by assumption converges to $\eps$ as $n\rarrow\infty$. By symmetry the same will hold for $\mu(\cup_{k\geq n} A_k\setminus gA_k)$. Since $\eps$ was arbitrary, this finishes the proof.
\end{proof}

\subsection{Cocycles}

We briefly recall the notion of (1-)cocycles, merely to fix notation, as they will only be used for the proof of Theorem~\ref{thm:R} in section~\ref{sec:R}. We refer to \cite[Chapter~20]{Kec10} for a more detailed exposition. Let $\Gam$ be a countable group with a (Borel) measure preserving action on a standard measure space $(X,\mu)$ and let $G$ be a Polish group. A (Borel) \emph{1-cocycle} for the action $\Gam\act X$ with values in $G$ is a Borel map $c:\Gam\times X\rarrow G$ satisfying the \emph{cocycle identity}
\[
c(gh,x) = c(g,h\cdot x) c(h,x),
\]
for all $g,h\in\Gam$ and $\mu$-almost every $x\in X$. Note that this in particular implies that $c(1,x)=1$ and $c(g,x)^{-1}=c(g^{-1},g\cdot x)$. We will identify cocycles that are equal almost everywhere and denote by $Z^1(\Gam\act X,G)$ the set of cocycles for the action $\Gam\act X$ with values in $G$. Observe that if $c\in Z^1(\Gam\act X,G)$ is independent of $x$, then it is given by a homomorphism $\Gam\rarrow G$. In particular when $G=S^1$ cocycles independent of $x$ are given by characters $\gam\in\hat{\Gam}$. Denote by $L(X,\mu,G)$ the space of all Borel maps $f:X\rarrow G$ (up to agreeing $\mu$-almost everywhere). Then we have an action $L(X,\mu,G)\act Z^1(\Gam\act X,G)$ given by
\[
(f\cdot c)(g,x) = f(g\cdot x)c(g,x) f(x)\inv.
\]
Two cocycles $c_1,c_2\in Z^1(\Gam\act X,G)$ are called \emph{cohomologous} if there exists $f\in L(X,\mu,G)$ such that $f\cdot c_1 = c_2$. If a cocycle $c$ is cohomologous to the trivial cocycle we call $c$ a \emph{coboundary}. We denote the set of coboundaries by $B^1(\Gam\act X,G)$.

\vspace{5pt}

Analogously, we can define cocycles for a countable Borel equivalence relation $E$, where countable means that the equivalence classes of $E$ are countable. For $C\sub E$ Borel we will write $C_x:=\{y\in X\mid (x,y)\in C\}$ and $C^y:=\{x\in X\mid (x,y)\in C\}$. Following \cite{FM77} we can define two $\sigma$-finite measures on $E\sub X^2$ by
\[
\nu_l(C) := \int_X \abs{C_x}\,d\mu(x) \quad\text{and}\quad \nu_r(C) := \int_X \abs{C^y}\,d\mu(y),
\]
where $\abs{S}$ denotes the cardinality of a set $S$. We say that $E$ is measure preserving if $\nu_l=\nu_r$ and in that case we denote this uniquely defined measure by $\nu$. A \emph{cocycle} of $E$ with values in $G$ is then defined to be a Borel map $c:E\rarrow G$ satisfying the \emph{cocycle identity}
\[
c(x,z)=c(y,z)c(x,y)
\]
for all $x,y,z\in Y$ belonging to a single $E$-class, and where $Y\sub X$ is an $E$-invariant Borel subset of $X$ of measure 1. Like before we identify two cocycles that are equal $\nu$-almost everywhere and we denote by $Z^1(E,G)$ the set of cocycles of $E$ with values in $G$. Here we have an action $L(X,\mu,G)\act Z^1(E,G)$ given by
\[
(f\cdot c)(x,y) = f(y)c(x,y)f(x)\inv
\]
and we call two cocycles $c_1,c_2\in Z^1(E,G)$ cohomologous if there exists $f\in L(X,\mu,G)$ such that $f\cdot c_1 = c_2$. The cocycles cohomologous to the trivial cocycle are again called coboundaries, and the set of coboundaries is denoted by $B^1(E,G)$.

\section{A structural result for Cartan subalgebras}\label{sec:structure}

For this section we fix the following.
\begin{itemize}
	\item $\Gam\act X$ a free ergodic pmp action of a countable group $\Gam$ on a standard probability space $(X,\mu)$,
	\item an arbitrary \II factor $N$,
	\item $\cM:=(\Linf\cross\Gam)\otb N$.
\end{itemize}
Note that
\[
\cM\cong (L^\infty(X)\otb N)\cross\Gamma
\]
where $\Gam$ acts trivially on $N$. Also, we can write $L^\infty(X)\otb N$ as the (constant) direct integral $L^\infty(X)\otb N=\dint_X N\,d\mu(x)$. For notational convenience we will drop the measure $\mu$ from the integrals from now on. In this section we will prove the main structural result, namely Theorem~\ref{thm:structure}. For this we will need a few lemmas. Note that only Lemma~\ref{lem:AuB} will assume that $\Gam\in\cC_{rss}$. The other results, in particular Lemma~\ref{lem:Cartaniff}, hold for any countable group $\Gam$.

\begin{lemma}\label{lem:meas}
	Suppose $f_1, f_2:X\rarrow \Cartan(N)$ are measurable functions such that $f_1(x)\sim_u f_2(x)$ in $N$ for almost every $x\in X$. Write $A_i=\dint_X f_i(x)\,dx$, $i=1,2$. Then there exists $u\in\cU(L^\infty(X)\otb N)$ such that $uA_1u^*=A_2$.
\end{lemma}
\begin{proof}
	We will prove that we can intertwine arbitrarily small corners of $A_1$ into arbitrarily small corners of $A_2$ inside $\LinfN$, after which we can conclude the proof with a maximality argument.
	\begin{claim}
		$A_1p_1\emb_{L^\infty(X)\otb N} A_2p_2$ for all nonzero projections $p_1\in A_1, p_2\in A_2$ such that $z(p_1)z(p_2)\neq 0$.
	\end{claim}
	\begin{proofofclaim}
		For $i=1,2$, we can write $p_i=\dint_X p_{i,x}\,dx$, with $p_{i,x}$ a projection in $A_{i,x}$. We then have the decompositions $A_ip_i=\dint_X (A_ip_i)_x\,dx = \dint_X A_{i,x}p_{i,x}\,dx$. Assume the claim doesn't hold. Then by Theorem~\ref{thm:Popa} we can find a sequence of unitaries $(u_n)_n\in\cU(A_1p_1)$ such that $\norm{E_{A_2p_2}(vu_nw)}_2\rarrow 0$ for all $v,w\in L^\infty(X)\otb N$. Writing $v=\dint_X v_x\,dx$, $w=\dint_X w_x\,dx$ and $u_n=\dint_X u_{n,x}\,dx$ we get
		\begin{align*}
		\norm{E_{A_2p_2}(vu_nw)}_2^2 &= \norm{E_{A_2p_2}\left(\dint_X v_xu_{n,x}w_x\,dx\right)}_2^2\\
		&= \norm{\dint_X E_{(A_2p_2)_x}(v_xu_{n,x}w_x)\,dx}_2^2\\
		&= \int_X \norm{E_{(A_2p_2)_x}(v_xu_{n,x}w_x)}_2^2\,dx,
		\end{align*}
		which by assumption converges to zero. Now let $\{t_i\}_{i\in\N}$ be a countable $\norm{.}_2$-dense subset of $N$. Letting $v$ and $w$ range over $\{1\ot t_i\mid i\in\N\}$ we get a subsequence of $(u_n)_n$ such that for almost every $x\in X$ we have $\norm{E_{(A_2p_2)_x}(t_iu_{n,x}t_j)}_2\rarrow 0$ for all $i,j\in\N$. Since $\{t_i\}_{i\in\N}$ is $\norm{.}_2$-dense in $N$, this means that $(A_1p_1)_x\not\emb_N (A_2p_2)_x$ for almost every $x$, which is absurd since $(A_1)_x=f_1(x)\sim_u f_2(x)=(A_2)_x$ in $N$ for almost every $x$, and $p_1$ and $p_2$ don't have disjoint central supports by assumption.
	\end{proofofclaim}

	\vspace{3pt}

	Now let $(p_i)_i$ and $(q_i)_i$ be maximal families of orthogonal projections such that $A_1p_i\sim_u A_2q_i$ inside $L^\infty(X)\otb N$ and put $p=\sum_i p_i$, $q=\sum_i q_i$. Then $A_1p\sim_u A_2q$ inside $L^\infty(X)\otb N$. In particular $p$ and $q$ are equivalent projections. Since $\LinfN$ is a finite von Neumann algebra, also $1-p$ and $1-q$ are equivalent, and so in particular have equal central support. Assuming $p\neq 1$ (and hence also $q\neq 1$), it then follows from the claim that $A_1(1-p)\emb_{L^\infty(X)\otb N} A_2(1-q)$. Moreover, by Lemma~\ref{lem:maxab}, $A_1$ and $A_2$ are maximal abelian inside $\LinfN$. Hence applying \cite[Theorem~C.3]{Va06} (see also \cite[Theorem~A.1]{Po01}) yields the existence of a partial isometry $v\in L^\infty(X)\otb N$ such that $v^*v\in A_1(1-p)$, $vv^*\in A_2(1-q)$ and $vA_1(1-p)v^*=A_2(1-q)vv^*$, contradicting the maximality above. We conclude that $A_1\sim_u A_2$ inside $L^\infty(X)\otb N$.
\end{proof}

\begin{lemma}\label{lem:Ax}
	Suppose $A$ is a Cartan subalgebra of $q(L^\infty(X)\otb N)q$ for some projection $q=\dint_X q_x\,dx \in L^\infty(X)\otb N$. Write $A=\dint_X A_x\,dx\sub \dint_X q_xNq_x\,dx$ and let $Y:=\{x\in X\mid q_x\neq 0\}$. Then $A_x$ is a Cartan subalgebra of $q_xNq_x$ for almost every $x\in Y$.
\end{lemma}
\begin{proof}
	First note that $A_x\sub q_xNq_x$ is maximal abelian for almost every $x\in Y$ by Lemma~\ref{lem:maxab}. To complete the proof we need to show that $A_x$ is regular in $q_xNq_x$ for almost every $x\in Y$. Therefore take $u\in\cN_{q(L^\infty(X)\otb N)q}(A)$ and write $u=\dint_X u_x\,dx$. Then
	\[
	\dint_X A_x\,dx = A = uAu^* = u\left(\dint_X A_x\,dx\right)u^* = \dint_X u_xA_xu_x^*\,dx
	\]
	and so $A_x=u_xA_xu_x^*$ almost everywhere. Let now $\{u^{(i)}\}$ be a countable $\norm{.}_2$-dense subset of $\cN_{q(L^\infty(X)\otb N)q}(A)$ and take null sets $E_i\sub X$ such that $A_x=u_x^{(i)}A_xu_x^{(i)*}$ for $x\notin E_i$. Put $E=\bigcup E_i$ so that $A_x=u_x^{(i)}A_xu_x^{(i)*}$ for all $i$ and all $x\notin E$. Thus for $x\notin E$ we have that $\{u_x^{(i)}\mid i\in\N\}\sub\cN_{q_xNq_x}(A_x)$ and so in order to prove that almost every $A_x$ is regular it suffices to show that $\{u_x^{(i)}\mid i\in\N\}\sub\cN_{q_xNq_x}(N_x)$ generates $q_xNq_x$ almost everywhere. For this note that if $T=\dint_X T_x\,dx\in q(L^\infty(X)\otb N)q = \{u^{(i)}\mid i\in\N\}''$ then $T_x\in\{u_x^{(i)}\mid i\in\N\}''$ almost everywhere by Propositions~\ref{prop:basicprop} and \ref{prop:decvNa}. Now take a countable $\norm{.}_2$-dense subset $\{t_n\}_n\in N$ and consider $q(1\ot t_n)q=\dint_X q_xt_nq_x\,dx$. The above claim implies that $q_xt_nq_x\in \{u_x^{(i)}\mid i\in\N\}''$ almost everywhere. Let $F_n:=\{x\in X\mid q_xt_nq_x\notin \{u_x^{(i)}\mid i\in\N\}''\}$ and $F:=\bigcup_n F_n$. Then for all $x$ not in the null set $F$ the $u_x^{(i)}$ generate $q_xNq_x$. We conclude that $A_x$ is regular in $q_xNq_x$ for almost every $x$, finishing the proof.
\end{proof}

\begin{lemma}\label{lem:Cartaniff}
	Suppose $A$ is a von Neumann subalgebra of $L^\infty(X)\otb N$ and write $A=\dint_X A_x\,dx$. Then the following are equivalent:
	\begin{enumerate}
		\item $A$ is a Cartan subalgebra of $\cM$,
		\item $A_x$ is a Cartan subalgebra of $N$ for almost every $x\in X$ and for every $g\in\Gam$, $A_{g^{-1}x}\sim_u A_x$ inside $N$ for almost every $x\in X$.
	\end{enumerate}
\end{lemma}
\begin{proof}
	$2\Rightarrow 1$. First note that $A:=\dint_X A_x\,dx$ is maximal abelian inside $\cM$. Indeed, by Lemma~\ref{lem:maxab} $A$ is maximal abelian inside $L^\infty(X)\otb N$. In particular $L^\infty(X)\sub A$. Since $L^\infty(X)$ is maximal abelian inside $L^\infty(X)\cross \Gam$ we then get
	\[
	A = A'\cap (\LinfN) \sub A'\cap \cM \sub A'\cap\Linf'\cap \cM = A'\cap (\LinfN) = A,
	\]	
	and so $A$ is maximal abelian inside $\cM$. To get regularity, we first observe that
	\[
	\cN_\cM(A)\super \cN_{L^\infty(X)\otb N}(A) = \dint_X \cN_N(A_x)\,dx
	\]
	where $\dint_X \cN_N(A_x)\,dx$ is the subset of $\LinfN$ consisting of all $u=\dint_X u_x\,dx$ such that $u_x\in\cN_N(A_x)$ for almost every $x\in X$. We note here that the set $\ClSub(\cU(N))$ of closed subsets of the Polish space $\cU(N)$ is a standard Borel space when given the Effros Borel structure. Moreover, in the proof of \cite[Proposition 12]{SV11} it is shown that the map $\Cartan(N)\rarrow \ClSub(\cU(N)):A\mapsto\cN_N(A)$ is Borel. Together with the measurability of the field $x\mapsto A_x$, this gives the measurability of the field $x\mapsto \cN_N(A_x)$.
	
	\begin{claim}
		$\cN_{L^\infty(X)\otb N}(A)''=\LinfN$.
	\end{claim}
	\begin{proofofclaim}
		Assume the claim doesn't hold. Then we can write $\cN_{L^\infty(X)\otb N}(A)''=\dint_X N_x\,dx$ where $N_x\subsetneq N$ for all $x$ in some subset of $X$ of positive measure. Let $\eps_x:=\sup\{\norm{u-E_{N_x}(u)}_2\mid u\in\cN_N(A_x)\}$ and define $Y:=\{x\in X\mid \eps_x>0\}$. By assumption $\mu(Y)>0$. Since $X\rarrow \ClSub(\cU(N)): x\mapsto \cN_N(A_x)$ and $X\rarrow \vNa(N):x\mapsto N_x$ are Borel, it follows easily that the set $\bB:= \{(x,u)\in Y\times \cU(N)\mid u\in\cN_N(A_x), \norm{u-E_{N_x}(u)}_2\geq \eps_x/2\}$ is a Borel subset of $Y\times \cU(N)$. By construction $\pi_1(\bB)=Y$, where $\pi_1$ is the projection onto the first coordinate. Hence by \cite[Theorem A.16]{Tak01} there exists a measurable function $\varphi:Y\rarrow \cU(N)$ such that $(x,\varphi(x))\in\bB$ for all $x\in Y$. Putting $\varphi(x)=1$ for $x\notin Y$, we then get $T:=\dint_X \varphi(x)\,dx\in \dint_X \cN_N(A_x)\,dx$, but $T\notin \dint_X N_x\,dx$, which is impossible.
	\end{proofofclaim}
	
	\vspace{3pt}
	
	It follows from the claim that $\LinfN\sub\cN_\cM(A)''$. Hence it suffices to argue that also $u_g\in \cN_\cM(A)''$ for every $g\in \Gam$. Fix $g\in \Gam$. For an element $a=\dint_X a_x\,dx\in L^\infty(X)\otb N$ we see that
	\[
	u_gau_g^* = (\sigma_g\ot\id)(a) = (\sigma_g\ot\id)\left(\dint_X a_x\,dx\right) = \dint_Xa_x\,d(gx) = \dint_X a_{g^{-1}x}\,dx,
	\]
	so $u_g$ acts by \textquotedblleft shifting" the components of the direct integral, yielding $u_gAu_g^*=A_{g^{-1}}$ where $A_{g^{-1}}:=\dint_X A_{g^{-1}x}\,dx$. On the other hand, from Lemma~\ref{lem:meas} and the given fact that $A_{g^{-1}x}\sim_u A_x$ for almost every $x\in X$, it follows that there exists $u(g)\in\cU(L^\infty(X)\otb N)$ such that $u(g)A_{g^{-1}}u(g)^*=A$. Putting them together yields
	\begin{equation}\label{eq:u(g)}
	(u(g)u_g)A(u(g)u_g)^*=u(g)A_{g^{-1}}u(g)^* = A.
	\end{equation}
	Since $u(g)\in L^\infty(X)\otb N$ and we already know that $L^\infty(X)\otb N\sub \cN_\cM(A)''$ we get $u_g\in\cN_\cM(A)''$ which finishes the proof of one implication.
	
	$1\Rightarrow 2$. First of all it follows from \cite{Dye63} that $A$, being Cartan in $\cM$, is Cartan in $L^\infty(X)\otb N$ as well (see also \cite{JP82}). So from Lemma~\ref{lem:Ax} it follows that $A_x$ is a Cartan subalgebra of $N$ for almost every $x\in X$ and we are left with showing that for every $g\in\Gam$, $A_{g^{-1}x}\sim_u A_x$ inside $N$ for almost every $x\in X$. Fix $g\in\Gam$. From \cite[Theorem~A.1]{Po01} it follows that it suffices to show that $A_{g^{-1}x}\emb_N A_x$. Since $A$ is a Cartan subalgebra of $\cM$ we know that $u_g\in\cN_{\cM}(A)''$. Since $\cN_{\cM}(A)$ is closed under products, this means that for a given $n\in\N$, we can find $\lam_1, \dots, \lam_m\in\C$ and $u_1, \dots, u_m\in\cN_{\cM}(A)$ such that
	\[
	\norm{\sum_{i=1}^m \lam_i u_i - u_g}^2_2\leq \frac{1}{n}.
	\]
	Writing $\sum_{i=1}^m \lam_i u_i=\sum_{h\in\Gam} c_h u_h$ with $c_h\in \LinfN$ we get that
	\[
	\frac{1}{n}\geq \norm{\sum_h c_h u_h - u_g}^2_2\geq \tau((c_g-1)^*(c_g-1)) = \int_X \tau_N((c_{g,x}-1_N)^*(c_{g,x}-1_N))\,dx,
	\]
	where we have written $c_g=\dint_X c_{g,x}\,dx$. In particular $c_{g,x}\neq 0$ on a set of measure at least $1-\frac{1}{n}$ and so for every such $x$ at least one of the $u_i=\sum_h b^i_h u_h$ satisfies $b_{g,x}^i\neq 0$ (where again $b^i_h=\dint_X b^i_{h,x}\,dx\in \LinfN$). Doing this for every $n\in\N$ implies that the set $E:=\{x\in X\mid \exists u=\sum_h b_h u_h\in\cN_\cM(A)\text{ such that } b_{g,x}\neq 0\}$ has full measure inside $X$. Now take $x\in E$ and $u=\sum_h b_h u_h\in\cN_\cM(A)$ such that $b_{g,x}\neq 0$. Let $\theta\in\Aut(A)$ be such that
	\begin{equation}\label{eq:uathau}
	ua=\theta(a)u, \qquad \text{for all } a\in A.
	\end{equation}
	Writing $a=\dint_X a_x\,dx$, $\theta(a)=\dint_X \theta(a)_x\,dx$ and substituting in equation \eqref{eq:uathau}, we find that for all $h\in\Gam$ we have $b_{h,x} a_{h^{-1}x} = \theta(a)_x b_{h,x}$ almost everywhere. In particular this holds for $g$. Using the polar decomposition for $b_{g,x}$ it then follows that there exists a nonzero partial isometry $v\in N$ such that
	\[
	v a_{g^{-1}x} = \theta(a)_x v, \qquad \text{for all } a\in A.
	\]
	This relation implies that $vA_{g\inv x}\sub A_x v$ and hence gives the desired intertwining $A_{g\inv x}\emb_N A_x$, finishing the proof of the other implication.
\end{proof}

\begin{lemma}\label{lem:AuB}
	Assume $\Gam\in\cC_{rss}$ and let $A\sub\cM$ be a Cartan subalgebra. Then $A$ is unitarily conjugate to a Cartan subalgebra of $\cM$ contained in $L^\infty(X)\otb N$.
\end{lemma}
\begin{proof}
	Applying the dichotomy in Definition~\ref{def:rss} for $\Gam\in\cC_{rss}$ we get that either $A\emb_{\cM} L^\infty(X)\otb N$ or $\cN_\cM(A)''=\cM$ is amenable relative to $L^\infty(X)\otb N$. Since $\Gam$ is non-amenable the latter is not possible (see for instance \cite[Proposition~2.4]{OP07}). We conclude that we can find projections $p\in A$, $q\in L^\infty(X)\otb N$, a nonzero partial isometry $v\in q\cM p$ and a $^*$-homomorphism $\psi: Ap\rarrow q(L^\infty(X)\otb N)q$ such that $\psi(a)v=va$ for all $a\in Ap$, $v^*v=p$, and $vv^*\in\psi(Ap)'\cap q\cM q$. Moreover, by \cite[Lemma~1.5]{Io11}, we can assume that $\psi(Ap)$ is maximal abelian in $q(L^\infty(X)\otb N)q$. In particular this means that $L^\infty(X)q\sub\psi(Ap)$ and taking relative commutants we get $\psi(Ap)'\cap q\cM q\sub (L^\infty(X)q)'\cap q\cM q=q(L^\infty(X)\otb N)q$. Since $\psi(Ap)$ is maximal abelian in $q(L^\infty(X)\otb N)q$ it then follows that
	\[
	\psi(Ap)'\cap q\cM q = \psi(Ap)'\cap q(L^\infty(X)\otb N)q = \psi(Ap).
	\]
	As $vv^*\in\psi(Ap)'\cap q\cM q$, we get $vv^*\in\psi(Ap)$ and so after possibly replacing $q$ by a subprojection we can assume that $vv^*=q$. Hence $\Ad(v):p\cM p\rarrow q\cM q$ is an isomorphism and since $Ap$ is a Cartan subalgebra of $p\cM p$ we conclude that $\psi(Ap)=vApv^*$ is a Cartan subalgebra of $q\cM q$.
	
	\begin{claim}
		There is a projection $p'\in Ap$ such that $\psi(Ap')\sub B$ where $B$ is a Cartan subalgebra of $\cM$ contained in $\LinfN$.
	\end{claim}
	\begin{proofofclaim}
		Note that the central trace of $q=\dint_X q_x\,dx\in L^\infty(X)\otb N$ is given by $\ctr(q)=\dint_X \tau_N(q_x)\,dx\in L^\infty(X)$. Proposition~7.17 in \cite{SZ79} tells us that for any $f\in L^\infty(X)$ with $f\leq \ctr(q)$, there exists a projection $q'\leq  q$ in $\psi(Ap)$ such that $\ctr(q')=f$. Also, since $q$ is nonzero, there exists $n\in \N$ such that $\mu(X_0)>0$, where
		\[
		X_0:=\{x\in X\mid \tau(q_x)\geq \frac{1}{n}\}.
		\]
		Consider now $f=\frac{1}{n}\een_{X_0}$ and take $q'\in\psi(Ap)$ such that $q'\leq q$ and $\ctr(q')=f$. In particular we have $\tau(q'_x)=\frac{1}{n}$ for $x\in X_0$ and $\tau(q'_x)=0$ otherwise.
		
		\vspace{3pt}
		
		Let now $p'=\psi^{-1}(q')$. Denoting the restriction of $\psi$ still by $\psi$ we get an injective homomorphism $\psi:Ap'\rarrow q'(L^\infty(X)\otb N)q'$ such that $\psi(Ap')=\dint_X \tilde{A}_x\,dx$ is a Cartan subalgebra of $q'\cM q'$. It then follows from Lemma~\ref{lem:Cartaniff} that $\tilde{A}_x$ is a Cartan subalgebra of $q'_xNq'_x$ for almost every $x\in X_0$. Put $p_1:=q'$ and consider $(1-p_1)(L^\infty(X_0)\otb N)(1-p_1)$. Applying again \cite[Proposition~7.17]{SZ79} we can find a projection $p_2\leq 1-p_1$ such that $\ctr(p_2)=\frac{1}{n}\een_{X_0}$. Continuing this we in the end find mutually orthogonal projections $p_1, \dots, p_n\in L^\infty(X_0)\otb N$ such that $\ctr(p_i)=\frac{1}{n}\een_{X_0}$ for every $i$.
		
		\vspace{3pt}
		
		In particular $p_1,\dots,p_n$ are equivalent and we can take partial isometries $u_i$ such that $u_i^*u_i=p_1$ and $u_iu_i^*=p_i$. Define
		\[
		\tilde{B}:=\oplus_{i=1}^n u_i\psi(Ap')u_i^*.
		\]
		Then $\tilde{B}=\dint_X \tilde{B}_x\,dx$ is a Cartan subalgebra of $L^\infty(X_0)\otb N$.
		
		\vspace{3pt}
		
		We want to upgrade this to a Cartan subalgebra of $\cM$. For this we note that the proof of the implication $1\Rightarrow 2$ in Lemma~\ref{lem:Cartaniff} goes through for a Cartan inclusion $A\sub q\cM q$ for some projection $q\in L^\infty(X)\otb N$. In the above situation this implies that we also had $\tilde{A}_{g^{-1}x}\emb_N \tilde{A}_x$ for almost every $x$ when both sides are nonzero and hence the same holds for $\tilde{B}_x$ instead of $\tilde{A}_x$.
		
		\vspace{3pt}
		
		Summarizing, we now have a measurable map $X_0\rarrow \Cartan(N): x\mapsto \tilde{B}_x$ such that $\tilde{B}_{g^{-1}x}\sim_u \tilde{B}_x$ inside $N$ for almost every $x\in X_0\cap gX_0$. Using the ergodicity of the action we can then extend this to a map
		\[
		X\rarrow \Cartan(N): x\mapsto B_x
		\]
		with the same properties. (A way to explicitly write this down is the following: Enumerate $\Gam:=\{e=g_1, g_2, \dots\}$ and send $x\in X$ to the Cartan subalgebra associated to $g_ix$ where $i$ is minimal among $\{n\in\N\mid g_nx\in X_0\}$.) Applying Lemma~\ref{lem:Cartaniff} again this implies that $B:=\dint_X B_x\,dx$ is a Cartan subalgebra of $\cM$ contained in $\LinfN$ which by construction contains $\psi(Ap')$, thus finishing the proof of the claim.
	\end{proofofclaim}

	\vspace{3pt}

	It follows from the claim that $A\emb_\cM B$, where $A$ and $B$ are both Cartan subalgebras of $\cM$. Hence $A$ and $B$ are unitarily conjugate by \cite[Theorem~A.1]{Po01}, which finishes the proof of the lemma.
\end{proof}

\begin{proof}[Proof of Theorem~\ref{thm:structure}]
	If $A$ is a Cartan subalgebra of $\cM$, then it follows from Lemma~\ref{lem:AuB} that $A$ is unitarily conjugate to a Cartan subalgebra $B$ of $\cM$ which is contained in $\LinfN$. It follows from Lemma~\ref{lem:Cartaniff} that $B$ has the desired form. The other implication is immediate from Lemma~\ref{lem:Cartaniff}.
\end{proof}

\section{A non-classifiability result for the equivalence relation of unitary conjugacy}\label{sec:nonclass}

Recall that for any Polish group $G$ acting continuously on a Polish space $Y$ we write $\cR(G\act Y)$ for its orbit equivalence relation. Consider a countable group $\Gam$, a standard probability space $(X,\mu)$ and a free ergodic pmp action $\Gam\act X$. Let $N$ be a \II factor and consider $\cM:=(L^\infty(X)\cross\Gam)\otb N$. Recall the equivalence relation $E_0$ on $\{0,1\}^\N$ where ${\bf x}E_0{\bf y}$ if there exists $N\in\N$ such that $x_n=y_n$ for all $n\geq N$. Theorem~\ref{thm:nonclass} says that if the relation of Cartan subalgebras of $N$ up to unitary conjugacy is not smooth and the action $\Gam\act X$ is not strongly ergodic, then the equivalence relation of Cartan subalgebras of $\cM$ up to unitary conjugacy cannot be classified by countable structures. The proof consists of two different parts. First of all we will use the structural results from section~\ref{sec:structure} to reduce the problem to studying the space of homomorphisms $\Hom(\cR(\Gam\act X),\cR(\cU(N)\act\Cartan(N)))$. Secondly we will prove that given any ergodic but not strongly ergodic pmp action $\Gam\act X$, the space of homomorphisms $\Hom(\cR(\Gam\act X),E_0)$ is not classifiable by countable structures. Using the Borel reduction $E_0\leq_B\cR(\cU(N)\act\Cartan(N))$ the desired result will then easily follow.

\begin{convention}
	All standard Borel spaces will come equipped with a Borel measure, and we will be identifying Borel functions almost everywhere.
\end{convention}

Given Polish spaces $X$, $Y$, and a measure $\mu$ on $X$, we write $B(X,Y)$ for the set of all Borel maps from $X$ to $Y$, identified $\mu$-almost everywhere. With the $\sigma$-algebra generated by the functions $f\mapsto \mu(A\cap f\inv(B))$ for $A\sub X$ and $B\sub Y$ Borel, this becomes a standard Borel space. We now first of all note that given two Borel equivalence relations $E$, $F$ on $X$, $Y$ respectively, also $\Hom(E,F)$ is a standard Borel space. One way to see this is the following. Since $E$ is by assumption a Borel subset of $X\times X$, hence a standard Borel space, $B(E,Y\times Y)$ is a standard Borel space as well. Identifying $\Hom(E,F)$, $B(E,F)$ and $B(X,Y)$ with the image of their canonical embedding into $B(E,Y\times Y)$, it is then easy to see that $\Hom(E,F)=B(E,F)\cap B(X,Y)$. We conclude that $\Hom(E,F)$ is Borel, being the intersection of two Borel sets.

\begin{lemma}\label{lem:eqrel}
	The equivalence relation $(\Hom(\cR(\Gam\act X),\cR(\cU(N)\act\Cartan(N))),\sim_u)$ is Borel reducible to $\cR(\cU(\cM)\act\Cartan(\cM))$, where for $\varphi,\psi\in \Hom(\cR(\Gam\act X),\cR(\cU(N)\act\Cartan(N))$ we let $\varphi\sim_u \psi$ if and only if $\varphi(x)\sim_u \psi(x)$ for almost every $x\in X$. Moreover, if $\Gam\in\cC_{rss}$, the aforementioned equivalence relations are Borel bi-reducible.
\end{lemma}
\begin{proof}
	We will write $\Cartan_{\LinfN}(\cM):=\{A\in\Cartan(\cM)\mid A\sub \LinfN\}$.
	
	\vspace{5pt}
	
	\emph{Step 1. $(\Hom(\cR(\Gam\act X),\cR(\cU(N)\act\Cartan(N))),\sim_u)\sim_B (\Cartan_{\LinfN}(\cM),\sim_u)$ where the latter equivalence relation is unitary conjugacy in $\cM$.}
	
	First of all we note that $\Cartan_{\LinfN}(\cM)$ is a standard Borel space. Indeed, we can write it as the intersection of $\Cartan(\cM)$ and $\vNa(\LinfN)$ inside $\vNa(\cM)$. The former is Borel by Proposition~\ref{prop:SV-Cartan} and the latter is Borel being the fixed points of the map $M\mapsto M\cap \LinfN$ which is Borel by \cite[Corollary~2]{Eff65}. Using Lemma~\ref{lem:Cartaniff}, given $A=\dint_X A_x\,dx\in \Cartan_{\LinfN}(\cM)$ we get a function $[f_A:x\mapsto A_x]\in \Hom(\cR(\Gam\act X),\cR(\cU(N)\act\Cartan(N)))$. Conversely, given such a function we can build a Cartan subalgebra $A=\dint_X f(x)\,dx$. Moreover, by Lemma~\ref{lem:meas}, we have that there exists $u\in\cU(\LinfN)$ such that $uAu^*=B$ if and only if $f_A\sim_u f_B$.
	
	Suppose now that we have $A,B\in \Cartan_{\LinfN}(\cM)$ and $u\in\cM$ such that $uAu^*=B$. To complete the proof of step 1, we need to show that we can replace $u$ by a unitary in $\LinfN$. Write $u=\sum_g b_gu_g$ with $b_g\in\LinfN$. Recall from \eqref{eq:u(g)} in the proof of $[2\Rightarrow 1]$ of Lemma~\ref{lem:Cartaniff}, that for every $g\in \Gam$ we can find a unitary $u(g)\in\cU(\LinfN)$ such that
	\[
	u_g^*u(g)^*Au(g)u_g = A.
	\]
	From the fact that $uAu^*=B$, we then get $uu_g^*u(g)^*Au(g) = Buu_g^*$. Applying the conditional expectation onto $\LinfN$ we get
	\[
	E_{\LinfN}(uu_g^*)[u(g)^*Au(g)]=B E_{\LinfN}(uu_g^*).
	\]
	Note that $E_{\LinfN}(uu_g^*)=b_g$, so by taking the polar decomposition we get a partial isometry $v\in\LinfN$ such that $vu(g)^*Au(g)=Bv$. Writing $v=\dint_X v_x\,dx$, $u(g)^*Au(g)=\dint_X [u(g)^*Au(g)]_x\,dx$ and $B=\dint_X B_x\,dx$ it follows that $v_x[u(g)^*Au(g)]_x=B_xv_x$. Since $u$ is a unitary we can find for almost every $x\in X$ an element $g\in \Gam$ such that $b_{g,x}\neq 0$, where we have written $b_g=\dint_X b_{g,x}\,dx$. Since $v_x\neq 0$ when $b_{g,x}\neq 0$, this means there exists for almost every $x\in X$ an element $g\in\Gam$ such that
	\[
	[u(g)^*Au(g)]_x \emb_N B_x.
	\]
	It follows that $A_x\emb_N B_x$ for almost every $x\in X$ and since $N$ is a factor we have $A_x\sim_u B_x$ inside $N$ for almost every $x\in X$. Applying Lemma~\ref{lem:meas} we thus get a unitary $u'\in\cU(\LinfN)$ such that $u'Au'^*=B$ and we conclude that for $A,B\in\Cartan_{\LinfN}(\cM)$, being unitarily conjugate inside $\cM$ is equivalent to being unitarily conjugate inside $\LinfN$.
	
	\vspace{5pt}
	
	\emph{Step 2. $(\Cartan_{\LinfN}(\cM),\sim_u)\leq_B \cR(\cU(\cM)\act\Cartan(\cM))$.}
	
	We have the inclusion map $i:\Cartan_{\LinfN}(\cM)\hookrightarrow \Cartan(\cM)$ for which we obviously have $A\sim_u B\iff i(A)\sim_u i(B)$.
	
	Step 1 and step 2 easily imply the first half of the lemma. For the moreover part, we conclude with the following.
	
	\vspace{5pt}

	\emph{Step 3. If $\Gam\in\cC_{rss}$, then $\cR(\cU(\cM)\act\Cartan(\cM))\leq_B (\Cartan_{\LinfN}(\cM),\sim_u)$.}
	
	For this direction we consider
	\begin{align*}
	\bB &:= \{(A,B)\in \Cartan(\cM)\times \Cartan_{\LinfN}(\cM)\mid A\sim_u B\}\\
	&= \{(A,B)\in \Cartan(\cM)\times \Cartan(\cM)\mid A\sim_u B\} \cap \Cartan(\cM)\times \Cartan_{\LinfN}(\cM)
	\end{align*}
	which is a Borel subset of $\Cartan(\cM)\times \Cartan_{\LinfN}(\cM)$ by Proposition~\ref{prop:SV-Cartan}(2) and the proof of step 1. Also, it follows from Lemma~\ref{lem:AuB} that $\pi_1(\bB)=\Cartan(\cM)$ where $\pi_1$ is the projection onto the first coordinate. Hence by \cite[Theorem~A.16]{Tak01} there exists a measurable function $\Psi:\Cartan(\cM)\rarrow\Cartan_{\LinfN}(\cM)$ such that $(A,\Psi(A))\in\bB$, i.e. $A\sim_u\Psi(A)$, for every $A\in\Cartan(\cM)$, which finishes step 3 and the proof of the moreover part of the lemma.	
\end{proof}

Consider now $\Hom(\cR(\Gam\act X),E_0)$ with the equivalence relation given by $\varphi\sim\psi$ if and only if $\varphi(x)E_0\psi(x)$ almost everywhere. We will prove the following more precise version of Theorem~\ref{thm:nonstrongE_01}.

\begin{theorem}\label{thm:nonstrongE0}
	Let $\Gam$ be a countable group, $(X,\mu)$ a standard probability space and $\Gam\act X$ an ergodic pmp action that is not strongly ergodic. Then there exists a map $f:(0,1)^\N\rarrow\Hom(\cR(\Gam\act X),E_0)$ satisfying all conditions from Proposition~\ref{prop:crit}. In particular $(\Hom(\cR(\Gam\act X),E_0),\sim)$ is not classifiable by countable structures.
\end{theorem}
\begin{proof}
	Consider the action $\oplus_\N \Z\act (\Pi_\N S^1,\nu)$ where the $i^\text{th}$ component of $\oplus_\N \Z$ acts on the $i^\text{th}$ component of $\Pi_\N S^1$ by irrational rotation by $\al_i\in\R\backslash\Q$ and where $\nu:=\ot_\N\,m$ for $m$ the normalized Lebesgue (probability) measure on $S^1$. For $r\in (0,1)$ define the subset $C_r\sub S^1$
	\[
	C_r:=\left\{ e^{2\pi i (r+t)}\mid t\in [0,1/2]\right\},
	\]
	i.e. $C_r$ is half of the circle starting from the point $e^{2\pi i r}$. In the following, given any measure space $Y$ we will write $\cB(Y)$ for the space of all Borel subsets of $Y$, and we will equip it with the pseudo-metric given by the measure of the symmetric difference. Consider the function
	\begin{equation}\label{eq:vphi}
	\varphi: (0,1)^\N\rarrow \{\text{sequences of Borel subsets of } \Pi_\N S^1\}=\prod_\N \cB(\prod_\N S^1):{\bf t}=(t_n)_n\mapsto (B^{\bf t}_n)_n
	\end{equation}
	where
	\[
	B^{\bf t}_n:=\underbrace{S^1\times\dots\times S^1}_{n-1}\times C_{t_n}\times S^1\times\dots.
	\]
	The sequences $(B^{\bf t}_n)_n$ are clearly almost invariant for the action of $\oplus_\N \Z$ and satisfy $\nu(B^{\bf t}_n)=\frac{1}{2}$ for every $n$. Also, for any element $a=(a_n)_n\in\oplus_\N \Z$ we have $\nu(aB^{\bf t}_n\Del B^{\bf t}_n)=0$ for $n$ sufficiently large. Moreover, we note that
	\begin{equation}\label{eq:measureBsBt}
	\nu(B^{\bf s}_n\Del B^{\bf t}_n)=m(C_{s_n}\Del C_{t_n})=2\min\{\abs{s_n-t_n},1-\abs{s_n-t_n}\}.
	\end{equation}
	
	Using Theorem~\ref{thm:JS} together with the lack of strong ergodicity we get a factor map from the action $\Gam\act X$ to the action of $\Z$ on some standard probability space. By \cite{Dye59} any two ergodic $\Z$-actions are orbit equivalent and so we also get a factor map $\theta:X\rarrow \Pi_\N S^1$ for the above actions of $\Gam$ and $\oplus_\N \Z$ (and in fact to any ergodic action of any amenable group on some standard probability space by \cite[Theorem~6]{OW80}). Lifting the almost invariant sequence $(B^{\bf t}_n)_n$ for $\oplus_\N \Z\act (\Pi_\N S^1,\nu)$ we get by Lemma~\ref{lem:aiseq} an almost invariant sequence $(A^{\bf t}_n)_n:=\theta^{-1}(B^{\bf t}_n)_n$ for the action $\Gam\act (X,\mu)$. Following \cite{JS87} we then consider the map
	\begin{equation}\label{eq:bfA}
	A^{\bf t}:X\rarrow \{0,1\}^\N: x\mapsto (\een_{A^{\bf t}_n}(x))_n.
	\end{equation}
	We claim that for all $g\in\Gam$ we have $A^{\bf t}(x)E_0A^{\bf t}(gx)$ for almost every $x\in X$. Indeed the set of $x\in X$ satisfying the claim has measure
	\begin{equation}\label{eq:goodxs}
	\mu(\{x\in X\mid \exists N\in\N, \forall n\geq N: x\notin gA^{\bf t}_n\Del A^{\bf t}_n\})=\mu\left(\bigcup_{N\in\N}\left(\bigcap_{n\geq N} X\backslash(gA^{\bf t}_n\Del A^{\bf t}_n)\right)\right)=1.
	\end{equation}
	For the last equality we used the moreover part of Lemma~\ref{lem:aiseq} and the fact that for every $a\in\oplus_\N \Z$, $\nu(aB^{\bf t}_n\Del B^{\bf t}_n)=0$ for $n$ sufficiently large to deduce that $\lim_n \mu(\cup_{k\geq n} gA^{\bf t}_k\Del A^{\bf t}_k)=0$. In terms of the equivalence relations this means that $A^{\bf t}$ is a homomorphism (on a co-null subset of $X$) from $\cR(\Gam\act X)$ to $E_0$. Recall that for two such elements we have $A^{\bf s}\sim A^{\bf t}$ if and only if $A^{\bf s}(x)E_0A^{\bf t}(x)$ for almost every $x\in X$.
	
	Altogether we now have a map $f:(0,1)^\N\rarrow \Hom(\cR(\Gam\act X),E_0): {\bf t}\mapsto f({\bf t}):=A^{\bf t}$. We claim that $f$ satisfies all the conditions from Proposition~\ref{prop:crit}. First of all we need that $f$ is a Borel map. For this consider $(0,1)^\N$ with the metric given by
	\[
	d((s_n)_n,(t_n)_n) = \sum_{n\in\N} 2^{-n}\abs{s_n-t_n},
	\]
	and the space $\Map(X,\{0,1\}^\N)$ of measurable functions from $X$ to $\{0,1\}^\N$ with the metric given by
	\[
	d(g,h)=\int_X \sum_{n\in\N} 2^{-n}\abs{g(x)_n-h(x)_n}\,dx.
	\]
	It is then easy to see that the map $(0,1)^\N\rarrow \Map(X,\{0,1\}^\N):{\bf t}\mapsto A^{\bf t}$ is actually continuous. The two conditions of Proposition~\ref{prop:crit} follow easily from the first part of the proof:
	
	\vspace{3pt}
	
	1. Suppose we have ${\bf s}$, ${\bf t}\in (0,1)^\N$ such that ${\bf s}-{\bf t}\in \ell^1$. Similar to \eqref{eq:goodxs} above, the points $x\in X$ such that $A^{\bf s}(x)E_0A^{\bf t}(x)$ are given by
	\[
	\{x\in X\mid \exists N\in\N, \forall n\geq N: x\notin A^{\bf s}_n\Del A^{\bf t}_n\}=\bigcup_{N\in\N}\left(\bigcap_{n\geq N} X\backslash(A^{\bf s}_n\Del A^{\bf t}_n)\right).
	\]
	By \eqref{eq:measureBsBt} we have $\mu(X\backslash(A^{\bf s}_n\Del A^{\bf t}_n)\geq 1-2\abs{s_n-t_n}$. Since ${\bf s}-{\bf t}\in \ell^1$ it follows easily that $A^{\bf s}(x)E_0A^{\bf t}(x)$ for almost every $x\in X$, i.e. $f({\bf s})\sim f({\bf t})$.
	
	\vspace{3pt}
	
	2. Suppose we have ${\bf s}$, ${\bf t}\in (0,1)^\N$ such that $\abs{s_n-t_n}\not\rarrow 0$ and $\abs{s_n-t_n}\not\rarrow 1$. Then $\mu(A^{\bf s}_n\Del A^{\bf t}_n)\not\rarrow 0$ by \eqref{eq:measureBsBt}, so we can find $\eps>0$ such that $\mu(A^{\bf s}_n\Del A^{\bf t}_n)\geq\eps$ for infinitely many $n$. Since $X$ has measure 1, this means we can find a set $Y$ of positive measure such that every element $y\in Y$ lies in infinitely many of the sets $A^{\bf s}_n\Del A^{\bf t}_n$ and thus $A^{\bf s}(y)\cancel{E}_0 A^{\bf t}(y)$. Combining this with the observation that any comeager subset of $(0,1)^\N$ contains elements ${\bf s}$ and ${\bf t}$ such that $\abs{s_n-t_n}\not\rarrow 0,1$ finishes the proof of the theorem.	
\end{proof}

\begin{proof}[Proof of Theorem~\ref{thm:nonclass}]
	Recall that $\varphi,\psi\in\Hom(\cR(\Gam\act X),\cR(\cU(N)\act\Cartan(N)))$ are equivalent if and only if $\varphi(x)\sim_u\psi(x)$ almost everywhere. Given the map $f:(0,1)^\N\rarrow\Hom(\cR(\Gam\act X),E_0)$ from Theorem~\ref{thm:nonstrongE0} we can use the given Borel reduction $\beta:\{0,1\}^\N\rarrow\Cartan(N)$ from $E_0$ to $\cR(\cU(N)\act\Cartan(N))$ to construct
	\[
	g:(0,1)^\N\rarrow\Hom(\cR(\Gam\act X),\cR(\cU(N)\act\Cartan(N))):{\bf t}\mapsto g({\bf t})=\beta(f({\bf t}))=\beta\circ A^{\bf t}.
	\]
	Since $\beta$ is a Borel reduction, it follows then immediately that $g$ satisfies all conditions from Proposition~\ref{prop:crit} as well. Together with Lemma~\ref{lem:eqrel}, this finishes the proof.	
\end{proof}

\begin{remark}
	More generally, we see that with exactly the same proof, one can get the following statement. Suppose we have a countable group $\Gam$, a standard probability space $(X,\mu)$ and an ergodic pmp action $\Gam\act X$ that is not strongly ergodic. Then for any non-smooth equivalence relation $E$ on a standard Borel space $Y$, the space of homomorphisms $\Hom(\cR(\Gam\act X),E)$ up to $E$-equivalence almost everywhere is not classifiable by countable structures.
\end{remark}

\begin{proof}[Proof of Theorem~\ref{thm:dichotomy}]
	It follows from Theorem~\ref{thm:nonclass} that $\cR(\cU(\cM)\act\Cartan(\cM))$ is not classifiable by countable structures when $\cR(\cU(N)\act\Cartan(N))$ is not smooth and $\Gam\act X$ is not strongly ergodic. Now assume $\cR(\cU(N)\act\Cartan(N))$ is smooth, i.e. there is a Borel map $f:\Cartan(N)\rarrow \R$ such that $C_1\sim_u C_2\iff f(C_1)=f(C_2)$. Then by steps 2 and 3 in Lemma~\ref{lem:eqrel} it is enough to show that $(\Cartan_{\LinfN}(\cM),\sim_u)$ is smooth. But given $A=\dint_X A_x\,dx\in\Cartan_{\LinfN}(\cM)$ we can associate to it the function $[x\mapsto f(A_x)]\in L(X,\mu,\R)$ which is a Polish space by \cite[ch.~19]{Kec10}. We already know that $A\sim_u B$ if and only if $A_x\sim_u B_x$ for almost every $x$. Hence $(\Cartan_{\LinfN}(\cM),\sim_u)$ is smooth, which finishes the proof.
\end{proof}

\section{Conjugacy by automorphisms}\label{sec:Aut}

In this section we aim for a non-classification result for Cartan subalgebras up to conjugacy by an automorphism similar to Theorem~\ref{thm:nonclass} for unitary conjugacy. More specifically we will prove the following theorem which will easily imply Theorem~\ref{thm:nonclassaut}.

\begin{theorem}\label{thm:nonclassaut2}
	Suppose $N$ is a \II factor such that
	\begin{itemize}
		\item $N$ has an irreducible regular amenable subfactor, i.e. an amenable subfactor $R\sub N$ such that $R'\cap N=\C 1$ and $\cN_N(R)''=N$,
		\item there exists a Borel map $\be:\{0,1\}^\N\rarrow \Cartan(N)$ such that
		\[
		xE_0y \quad \iff \quad \be(x)\sim_u\be(y) \quad \iff \quad \be(x)\sim_{sa}\be(y).
		\]
	\end{itemize}
	Let $\Gam\in\cC_{rss}$, $(X,\mu)$ be a standard probability space and $\Gam\act X$ be a free ergodic pmp action that is not strongly ergodic. Then the equivalence relation of Cartan subalgebras of $\cM=(\Linf\cross \Gam)\otb N$ up to conjugacy by an automorphism is not classifiable by countable structures.
\end{theorem}

For the proof of Theorem~\ref{thm:nonclassaut2} we start off with the following lemma.

\begin{lemma}\label{lem:LinfembalLinf}
	In the setting of Theorem~\ref{thm:nonclassaut2}, suppose $A$ and $B$ are Cartan subalgebras of $\cM$ contained in $\LinfN$ such that $\al(A)=B$ for some $\al\in\Aut(\cM)$. Then $\Linf\emb_{\al(\LinfN)}^s\al(\Linf)$.
\end{lemma}
\begin{proof}
	 Firstly note that the conclusion of the lemma is actually possible as $\Linf\sub\al(\LinfN)$. Indeed $\Linf\sub B$ since $B$ is a Cartan subalgebra of $\LinfN$ whose center is $\Linf$, and $B=\al(A)\sub\al(\LinfN)$. Now consider 
	\[
	\al(\Linf\otb R) \sub (\LinfN)\cross\Gam.
	\]
	Since the left hand side is amenable, applying the dichotomy for $\Gam\in\cC_{rss}$ we get that either $\al(\Linf\otb R) \emb_\cM \LinfN$ or $\cN_\cM(\al(\Linf\otb R))''$ is amenable relative to $\LinfN$. As this normalizer equals $\cM$ by assumption and $\Gam$ is not amenable, the latter is not possible (see \cite[Proposition~2.4]{OP07}). So we get that
	\[
	\al(\Linf\otb R) \emb_\cM \LinfN.
	\]
	Since $R'\cap N=\C 1$, taking relative commutants and using \cite[Lemma~3.5]{Va07} this implies that
	\[
	\Linf \emb_\cM \al(\Linf).
	\]
	Since $\cN_\cM(\Linf)''=\cM$ and $\cM$ is a factor, \cite[Lemma~2.4(3)]{DHI16} implies that we actually have
	\begin{equation}\label{eq:strongemb0}
	\Linf\emb_\cM^s\al(\Linf).
	\end{equation}
	We will upgrade this to an embedding inside $\al(\LinfN)$, i.e.
	\begin{equation}\label{eq:strongemb}
	\Linf\emb_{\al(\LinfN)}^s\al(\Linf),
	\end{equation}
	which would finish the proof of the lemma. Assume \eqref{eq:strongemb} does not hold. Then there exists a projection $p\in\Linf'\cap\al(\LinfN)$ such that $\Linf p \not\emb_{\al(\LinfN)} \al(\Linf)$. By Theorem~\ref{thm:Popa} we can then find $u_n\in\cU(\Linf p)$ such that
	\[
	\norm{E_{\al(\Linf)}(x^*u_ny)}_2\rarrow 0
	\]
	for all $x,y\in p\al(\LinfN)$. We will show that in this case the same holds for all $x,y\in p\cM$, contradicting \eqref{eq:strongemb0}. By density, we can assume that $x^*=\al(au_g)p$ and $y=p\al(bu_h)$ where $a,b\in\LinfN$ and $g,h\in\Gam$. Since $\al(\Linf)\sub \al(\LinfN)$ we have $E_{\al(\Linf)}=E_{\al(\Linf)}\circ E_{\al(\LinfN)}$. Using the fact that $\cU(\al(N))$ normalizes $\al(\Linf)$ we then get
	\begin{align*}
	\norm{E_{\al(\Linf)}(\al(au_g)u_n\al(bu_h)}_2 &=
	\norm{\al(u_g)E_{\al(\Linf)}(\al(u_g^*au_g)u_n\al(bu_{hg}))\al(u_g^*)}_2\\
	&= \norm{E_{\al(\Linf)}(E_{\al(\LinfN)}(\al(u_g^*au_g)u_n\al(b)\al(u_{hg})))}_2\\
	&= \norm{E_{\al(\Linf)}(\al(u_g^*au_g)u_n\al(b)\al(u_{hg})\del_{hg,e})}_2.
	\end{align*}
	For the last equality we used the fact that $\al(u_g^*au_g)u_n\al(b)\in\al(\LinfN)$ to take it out of $E_{\al(\LinfN)}$ together which the fact that the $u_g$'s for $g\neq e$ are orthogonal to $\LinfN$. Now the last line converges to $0$ by the discussion above, finishing the proof of the lemma.
\end{proof}

\begin{lemma}\label{lem:q}
	In the setting of Lemma~\ref{lem:LinfembalLinf}, there exists a nonzero projection $q\in B$ such that $\al(\Linf)q=\Linf q$.
\end{lemma}
\begin{proof}
	Let $p\in\LinfN\cap\al(\LinfN)$ be any nonzero projection. From Lemma~\ref{lem:LinfembalLinf} it follows that there exist $p_0\in\Linf p$, $p_1\in\al(\Linf)$, a $^*$-homomorphism $\psi: \Linf p_0\rarrow\al(\Linf)p_1$ and a nonzero partial isometry $v\in p_1\al(\LinfN)p_0$ such that $\psi(x)v=vx$ for all $x\in \Linf p_0$. Note however that $v$ in this case commutes with $\al(\Linf)$, giving $v\psi(x)=vx$. Multiplying on the left by $v^*$ and writing $q':=v^*v\in (\Linf p_0)'\cap p_0\al(\LinfN)p_0 = p_0[(\LinfN)\cap \al(\LinfN)]p_0$, we get 
	\begin{equation}\label{eq:ineq1}
	\Linf q'\sub\al(\Linf) q'.
	\end{equation}
	Repeating the same arguments for $\al^{-1}$, Lemma~\ref{lem:LinfembalLinf} tells us that also $\al(\Linf)\emb^s_{\LinfN}\Linf$. In particular $\al(\Linf) q'\emb_{\LinfN}\Linf$ and with the same reasoning as above in the proof of this lemma, we get $0\neq q\in (\LinfN)\cap \al(\LinfN)$ such that $q\leq q'$ and $\al(\Linf)q\sub\Linf q$. Together with \eqref{eq:ineq1} this means
	\begin{equation}\label{eq:equalq}
	\al(\Linf)q=\Linf q.
	\end{equation}
	Applying $E:=E_{\Linf\vee\al(\Linf)}$ we get the same equality with $E(q)$ instead of $q$. Multiplying with $f_n(E(q))$ where $f_n(t)=t\inv\een_{[1/n,\infty)}(t)$ and taking the limit as $n\rarrow\infty$, we also get the same for the support $s(E(q))$ of $E(q)$. Since $s(E(q))\in \Linf\vee\al(\Linf)\sub B$ this means we can assume that $q$ in equation \eqref{eq:equalq} belongs to $B$, which finishes the proof of the lemma.
\end{proof}

\begin{proof}[Proof of Theorem~\ref{thm:nonclassaut2}]
	We proceed in two steps. In the first step we show that, given two Cartan subalgebras of $\cM$ contained in $\LinfN$ that are conjugate by an automorphism of $\cM$, there are positive measure subsets of $X$ on which the Cartan subalgebras appearing in the respective integral decompositions are conjugate by a stable automorphism of $N$. In the second step we construct a map $f:(0,1)^\N\times (0,1)^\N\rarrow \Cartan(\cM)$ similar to the one in the proof of Theorem~\ref{thm:nonstrongE0}. We then use the proof of Theorem~\ref{thm:nonstrongE0} together with step 1 to verify that $f$ satisfies the conditions of Proposition~\ref{prop:crit}.
	
	\vspace{5pt}
	
	\textit{Step 1.} Suppose that $A=\dint_X A_x\,dx$ and $B=\dint_X B_x\,dx$ are Cartan subalgebras of $\cM$ contained in $\LinfN$ (cf. Theorem~\ref{thm:structure}) that are conjugate by an automorphism of $\cM$. We will show that there exists a nonsingular Borel isomorphism $\Phi:Y_2\rarrow Y_1$ between positive measure subsets of $X$ such that $A_x$ and $B_{\Phi^{-1}x}$ are conjugate by a stable automorphism of $N$ for all $x\in Y_1$. Take $\al\in\Aut(\cM)$ such that $\al(A)=B$. Using Lemma~\ref{lem:q}, we can take $q\in B$ such that $\al(\Linf)q=\Linf q$. Taking relative commutants we get
	\[
	q\al(\LinfN)q=q(\LinfN) q.
	\]
	Writing $\tilde{q}=\al^{-1}(q)\in A$, $Y_1:=\{x\in X\mid \tilde{q}_x\neq 0\}$, and $Y_2:=\{x\in X\mid q_x\neq 0\}$ this gives the following diagram, where applying $\al$ gives an isomorphism from the first row to the second at every level.
	
	\vspace{5pt}
	
	\centerline{\xymatrix{L^\infty(Y_1)\tilde{q} = \Linf \tilde{q} &\sub & A\tilde{q} &\sub & \tilde{q}(\LinfN)\tilde{q} = \dint_{Y_1} \tilde{q}_xN\tilde{q}_x\,dx & \ar@/^1pc/[d]^\al_\cong\\
			L^\infty(Y_2)q=\Linf q &\sub & Bq &\sub & q(\LinfN)q = \dint_{Y_2} q_xNq_x\,dx &}}
	
	\vspace{5pt}
	
	Possibly ignoring a zero measure subset, it then follows from Theorem~\ref{thm:Tak-Desint} that we have a nonsingular Borel isomorphism $\Phi:Y_2\rarrow Y_1$ and isomorphisms
	\[
	\al_x: \tilde{q}_xN\tilde{q}_x\xrightarrow{\sim} q_{\Phi^{-1}x}Nq_{\Phi^{-1}x}
	\]
	for every $x\in Y_1$. In particular
	\[
	\al_x(A_x\tilde{q}_x)=B_{\Phi^{-1}x}q_{\Phi^{-1}x},
	\]
	i.e. $A_x$ is conjugate to $B_{\Phi^{-1}x}$ by a stable automorphism of $N$ for every $x\in Y_1$.
	
	\vspace{5pt}
	
	\textit{Step 2.} Recall from the proof of Theorem~\ref{thm:nonstrongE0} the sets $C_r:=\left\{ e^{2\pi i (r+t)}\mid t\in [0,1/2]\right\}\sub S^1$ for $r\in (0,1)$. Consider the map (cf. $\varphi$ in \eqref{eq:vphi})
	\[
	\phi:(0,1)^\N\times (0,1)^\N\rarrow \prod_\N \cB(\prod_\N S^1):(s_n,t_n)_n\mapsto (D^{\bf s,t}_n)_n
	\]
	where
	\begin{align*}
	D^{\bf s,t}_{2n-1} &:= (S^1)^{n-1} \times C_{s_n} \times S^1 \times S^1\times \dots,\\
	D^{\bf s,t}_{2n} &:= (S^1)^{n-1} \times C_{t_n} \times S^1 \times S^1\times \dots.
	\end{align*}
	Using the given Borel map $\be:\{0,1\}^\N\rarrow\Cartan(N)$, we can associate a Cartan subalgebra of $\cM$ to every such sequence as follows (cf. Theorem~\ref{thm:nonstrongE0}). Using the non-strong ergodicity of $\Gam\act X$ we again get a factor map $\theta:X\rarrow \Pi_\N S^1$ for the actions of $\Gam$ and $\oplus_\N \Z$. Let $A^{\bf s,t}_n:=\theta\inv(D^{\bf s,t}_n)$. Then we can consider the map
	\[
	A^{\bf s,t}: X\rarrow \{0,1\}^\N:x\mapsto (\een_{A^{\bf s,t}_n}(x))_n.
	\]
	Arguing as in the proof of Theorem~\ref{thm:nonstrongE0} we get in this way a map
	\[
	(0,1)^\N\times (0,1)^\N\rarrow \Hom(\cR(\Gam\act X),\cR(\cU(N)\act\Cartan(N))):({\bf s,t})\mapsto \beta\circ A^{\bf s,t}.
	\]
	As in the proof of Step 1 of Lemma~\ref{lem:eqrel} we can associate a Cartan subalgebra of $\cM$ to every homomorphism $\psi$ from $\cR(\Gam\act X)$ to $\cR(\cU(N)\act\Cartan(N))$ via $A:=\dint_X \psi(x)\,dx$. Altogether this gives us a measurable map
	\[
	f:(0,1)^\N\times (0,1)^\N\rarrow \Cartan(\cM):({\bf s,t})\mapsto \dint_X\beta\circ A^{\bf s,t}(x)\,dx.
	\]
	We claim that $f$ satisfies the conditions in Proposition~\ref{prop:crit} for the relation of conjugacy by automorphisms on $\Cartan(\cM)$. The same proof as in Theorem~\ref{thm:nonstrongE0} goes through to show that $f({\bf s,t})$ and $f({\bf v,w})$ are unitarily conjugate (and so certainly conjugate by an automorphism) whenever ${\bf s}-{\bf v}\in \ell^1$ and ${\bf t}-{\bf w}\in \ell^1$. Hence the first condition follows immediately.
	
	For the second condition, take any comeager set $C\sub (0,1)^\N\times (0,1)^\N$. Then we can find sequences $(s_n)_n$, $(t_n)_n$, $(v_n)_n$ and $(w_n)_n$ in $(0,1)^\N$ such that $(s_n,t_n)_n\in C$, $(v_n,w_n)_n\in C$, $s_n,t_n,v_n\rarrow 0$ and $w_n\rarrow \frac{1}{2}$. Indeed, being comeager, $C$ is dense and so in particular intersects arbitrary neighbourhoods of $(0,0)$ and $(0,\frac{1}{2})$. We claim that in this case $A:=f({\bf s,t})$ and $B:=f({\bf v,w})$ are not conjugate by an automorphism of $\cM$, which would imply the second condition of Proposition~\ref{prop:crit}. So assume $A\sim_a B$. Then from step 1 we get positive measure subsets $Y_{1,2}\sub X$ and a partial automorphism $\Phi:Y_2\rarrow Y_1$ such that
	\[
	A_x \sim_{sa} B_{\Phi\inv(x)}
	\]
	for all $x\in Y_1$. Looking at the construction of $f$ above this means that we have
	\[
	\een_{(A_n)_n}(x) E_0 \een_{(B_n)_n}(\Phi\inv(x))
	\]
	for $x\in Y_1$. Here $(A_n)_n$ and $(B_n)_n$ are the almost invariant sequences used to construct $A$ respectively $B$, i.e. $A_n=\theta\inv(D^{\bf s,t}_n)$ and $B_n=\theta\inv(D^{\bf v,w}_n)$ where $\theta:X\rarrow \Pi_\N S^1$ is the factor map as before. Rephrasing the above, we have that for $x\in Y_1$,
	\[
	x\in A_n\cap Y_1 \quad \iff \quad \Phi\inv(x)\in B_n\cap Y_2 \quad \iff \quad x\in\Phi(B_n\cap Y_2)
	\]
	for $n$ large enough. Hence by shrinking $Y_1$ to a positive measure subset if needed, we can assume that there exists $n_0\in\N$ such that for all $n\geq n_0$
	\begin{equation}\label{eq:AB}
	A_n\cap Y_1 = \Phi(B_n\cap Y_2).
	\end{equation}
	Now recall that by construction we have
	\begin{align*}
	A_{2n-1} &:= \theta\inv((S^1)^{n-1} \times C_{s_n} \times S^1 \times \dots),\\
	A_{2n} &:= \theta\inv((S^1)^{n-1} \times C_{t_n} \times S^1 \times \dots),\\
	B_{2n-1} &:= \theta\inv((S^1)^{n-1} \times C_{v_n} \times S^1 \times \dots),\\
	B_{2n} &:= \theta\inv((S^1)^{n-1} \times C_{w_n} \times S^1 \times \dots).
	\end{align*}
	Together with \eqref{eq:AB} we get
	\[
	\{ a\in Y_1\mid \theta(a)_n\in C_{s_n}\} = A_{2n-1}\cap Y_1 = \Phi(B_{2n-1}\cap Y_2) = \Phi(\{ b\in Y_2\mid \theta(b)_n\in C_{v_n}\})
	\]
	and
	\[
	\{ a\in Y_1\mid \theta(a)_n\in C_{t_n}\} = A_{2n}\cap Y_1 = \Phi(B_{2n}\cap Y_2) = \Phi(\{ b\in Y_2\mid \theta(b)_n\in C_{w_n}\})
	\]
	for $n\geq n_0$. Since $\abs{s_n-t_n}\rarrow 0$, we have $\mu(A_{2n-1}\Del A_{2n})\rarrow 0$ (cf. \eqref{eq:measureBsBt}) and the above equalities then yield $\mu(\Phi(B_{2n-1}\cap Y_2)\Del\Phi(B_{2n}\cap Y_2))\rarrow 0$. Since $Y_2$ has positive measure, this contradicts the fact that $\abs{v_n-w_n}\rarrow \frac{1}{2}$, finishing the proof.
\end{proof}

Theorem~\ref{thm:nonclassaut2} allows us to get the first family of \II factors whose Cartan subalgebras up to conjugacy by an automorphism are not classifiable by countable structures. More specifically we can use the \II factors of Speelman and Vaes constructed in \cite{SV11} as follows. Let $G$ be a countable group, $K$ a compact abelian group and $\Lam<K$ a countable dense subgroup. Consider the \II factor
\begin{equation}\label{eq:N}
N:=L^\infty(K^G)\cross (G\times \Lam)
\end{equation}
where $G\act K^G$ is the Bernoulli action and $\Lam\act K^G$ acts via $\lam\cdot(k_g)_{g\in G}=(\lam k_g)_{g\in G}$. If $K_1<K$ is a closed subgroup such that $\Lam_1:=\Lam\cap K_1$ is dense in $K_1$, the subalgebra $\cC(K_1):= L^\infty(K^G/K_1)\cross \Lam_1$ is a Cartan subalgebra of $N$ (\cite[Lemma~6]{SV11}). When $G$ is a property (T) group such that $[G,G]=G$, \cite[Theorem~1]{SV11} tells us when two such Cartan subalgebras are unitarily conjugate or conjugate by a (stable) automorphism. This is then used to show in \cite[Theorem~2]{SV11} that for a specific choice of $K$ and $\Lam$ one gets a Borel reduction $\beta$ of $E_0$ into $\Cartan(N)$ for either the relation of being unitarily conjugate, conjugate by an automorphism or conjugate by a stable automorphism, i.e.
\begin{equation}\label{eq:SV11}
xE_0y \quad \iff \quad \beta(x)\sim_u\beta(y) \quad \iff \quad \beta(x)\sim_a\beta(y) \quad \iff \quad \beta(x)\sim_{sa}\beta(y).
\end{equation}
Moreover one can easily check that $(L^\infty(K^G)\cross\Lam)\sub N$ is an irreducible regular amenable subfactor. Hence the \II factor constructed in \cite[Theorem~2]{SV11} satisfies both conditions in Theorem~\ref{thm:nonclassaut2}. This now immediately implies Theorem~\ref{thm:nonclassaut}.

\begin{remark}
	Another approach would be to try to use the dichotomy for groups $\Gam\in\cC_{rss}$ as follows. Take an arbitrary \II factor $N$ and suppose we have $A,B\in\Cartan((\Linf\cross\Gam)\otb N)$ and $\al\in\Aut((\Linf\cross\Gam)\otb N)$ such that
	\[
	\al(A)=B.
	\]
	Writing $(L^\infty(X)\otb N)\cross\Gam = \al(L^\infty(X)\cross\Gam) \otb \al(N)$ and using the fact that $\Gam\in\cC_{rss}$, it follows from \cite[Lemma 5.2]{KV15} that either $\al(L^\infty(X)\cross\Gam)\emb_\cM L^\infty(X)\otb N$ or $\al(N)$ is amenable relative to $L^\infty(X)\otb N$. Using once more that $\Gam\in\cC_{rss}$, the latter implies that $\al(N)\emb_\cM L^\infty(X)\otb N$ (since $\cN_\cM(\al(N))''=\cM$ is not amenable relative to $L^\infty(X)\otb N$). Hence either
	\begin{equation}\label{eq:1}
	\al(L^\infty(X)\cross\Gam)\emb_\cM L^\infty(X)\otb N, 
	\end{equation}\label{eq:2}
	or
	\begin{equation}
	\al(N)\emb_\cM L^\infty(X)\otb N.
	\end{equation}
	Given the latter, it is not difficult to get to the same conclusion as in Lemma~\ref{lem:q}. So if we can in some way exclude the first possibility, this could give an alternative way to our conclusion, avoiding to use the specific requirements on $N$ in Theorem~\ref{thm:nonclassaut2}.
\end{remark}

\section{\texorpdfstring{The hyperfinite \II factor}{The hyperfinite II1 factor}}\label{sec:R}

In \cite{Pa85}, J. Packer explicitly constructs an uncountable family of Cartan subalgebras of the hyperfinite \II factor $R$ no two of which are unitarily conjugate. Combining the idea behind this construction with a turbulence result for cocycles from \cite{Kec10} we will directly show that the Cartan subalgebras of $R$ up to unitary conjugacy are in fact not classifiable by countable structures.

Recall that a pmp action $\Gam\act (X,\mu)$ is called \emph{compact} (or \textquotedblleft has pure point spectrum" in the formulation of \cite{Pa85}) if the image of $\Gam$ in $\Aut(X,\mu)$ is precompact in the weak topology (the smallest topology making the maps $T\mapsto \mu(T(A)\Del B)$ continuous for all measurable sets $A,B\sub X$).

We first restate two results from \cite{Pa85} used to construct a big family of Cartan subalgebras of $R$.

\begin{theorem}[{\cite[Corollary~2.6]{Pa85}}]\label{thm:LGCartan}
	Let $\Gam$ be a countable discrete abelian group, $(Y,\nu)$ a probability space and suppose $\Gam\curvearrowright Y$ is a free ergodic pmp action. Then $L\Gam$ is a Cartan subalgebra of $L^\infty(Y)\cross \Gam$ if and only if the action is compact.
\end{theorem}

Given a free ergodic pmp action $\Gam\act Y$ as in the theorem and a Borel 1-cocycle $c:\Gam\times Y\rarrow S^1$ we can construct an automorphism $A_c$ of $L^\infty(Y)\cross \Gam$ given by
\[
A_c(\sum_g a_gu_g) = \sum_g a^{(c)}_g u_g,
\]
where $a^{(c)}_g(y)=c(g,y)a_g(y)$. If the action is also compact, we get in this way a family of Cartan subalgebras $(A_c(L\Gam))\sub L^\infty(Y)\cross \Gam$ where $c$ ranges over all 1-cocycles. The following result allows us to tell these Cartan subalgebras apart.

\begin{theorem}[{\cite[Theorem~3.8]{Pa85}}]\label{thm:AalconjLG}
	Let $\Gam$ and $Y$ be as above and suppose $\Gam\curvearrowright Y$ is a free ergodic compact pmp action. Let $c:\Gam\times Y\rarrow S^1$ be a Borel 1-cocycle. Then $A_c(L\Gam)$ is unitarily conjugate to $L\Gam$ inside $L^\infty(Y)\cross \Gam$ if and only if $c$ is cohomologous to a cocycle of the form $\gam(g,y)=\gam(g)$ for some $\gam\in\hat{\Gam}$.
\end{theorem}

We will combine this theorem with the following results from \cite{Kec10}. Recall that for a Polish group $G$ and a standard measure space $(Y,\nu)$, we let $L(Y,\nu,G)$ be the space of all Borel maps $f:Y\rarrow G$ up to agreeing $\nu$-almost everywhere.

\begin{theorem}[{\cite[Corollary~27.4]{Kec10}}]\label{thm:turbcohom}
	Suppose $E$ is an ergodic equivalence relation that is not strongly ergodic and let $G\neq \{1\}$ be a Polish group admitting an invariant metric. Then the action of $L(Y,\nu,G)$ on $\overline{B^1(E,G)}$ is turbulent. In particular, the cohomology relation on $\overline{B^1(E,G)}$ (and thus also on $Z^1(E,G)$) does not admit classification by countable structures.
\end{theorem}

\begin{proposition}[{\cite[Corollary~22.2 and Proposition~23.5]{Kec10}}]\label{prop:orbits}
	Let $G$ be a Polish group admitting an invariant metric and let $G$ act continuously by isometries on a Polish metric space $(M,\rho)$. Assume that the orbit $G\cdot x$ is not closed. Then the action of $G$ on the invariant closed set $\overline{G\cdot x}$ is minimal and every orbit contained in $\overline{G\cdot x}$ is meager in $\overline{G\cdot x}$.
\end{proposition}

We can now give a short proof of Theorem~\ref{thm:R}, establishing that the Cartan subalgebras of the hyperfinite \II factor $R$ up to unitary conjugacy are not classifiable by countable structures.

\begin{proof}[Proof of Theorem~\ref{thm:R}]
	Following \cite[section~4]{Pa85} we consider $\Z_2$ with the $(1/2,1/2)$-measure. Let $(Y,\nu)$ be $\Pi_{i\in\N} \Z_2$ with product measure and put $\Gam:=\oplus_{i\in\N} \Z_2$. Then both $Y$ and $\Gam$ are topological groups for addition modulo 2 and $\Gam$ embeds as a countable dense subgroup in $Y$. It is easy to check that the translation action $\Gam\curvearrowright Y$ is free, ergodic, measure preserving and compact. Also we note that $L^\infty(Y)\cross \Gam$ can be identified with $\otb_\N (L^{\infty}(\Z_2)\rtimes \Z_2) \cong\otb_\N M_2(\C)\cong R$ and that $\hat{\Gam}\cong Y$. From Theorem~\ref{thm:LGCartan} it now follows that $L\Gam$ is a Cartan subalgebra of $L^\infty(Y)\cross \Gam\cong R$, and so we can consider all Cartan subalgebras of the form $A_c(L\Gam)$ for $c\in Z^1(\Gam\act Y,S^1)$. Note that Theorem~\ref{thm:AalconjLG} implies that $A_c(L\Gam)$ is unitarily conjugate to $A_d(L\Gam)$ if and only if $c^{-1}d$ is cohomologous to a cocycle $\gam\in\hat{\Gam}$. Consider now the action of $\hat{\Gam}\times L(Y,\nu,S^1)$ on $\overline{B^1(\Gam\act Y,S^1)}=Z^1(\Gam\act Y,S^1)$ (see \cite[Theorem~26.4]{Kec10} for the equality) given by
	\[
	(\gam,f)\cdot c(g,y) = \gam(g) f(gy) c(g,y) f(y)^{-1}.
	\]
	From the above it follows that the orbit equivalence relation of this action is Borel reducible to the equivalence relation of unitary conjugacy on $\Cartan(R)$. Hence it suffices to show that the above action is turbulent. However, we know that the action of just the $L(Y,\nu,S^1)$-part is turbulent by Theorem~\ref{thm:turbcohom}. So if we can show that the orbits for the action of $\hat{\Gam}\times L(Y,\nu,S^1)$ are still meager, the result will follow (since the other parts in the definition of turbulence are obviously satisfied). For this, note that the action is not transitive. Since the orbit of 1 is dense, it then immediately follows from Proposition~\ref{prop:orbits} (and the remark below) that all orbits are meager, finishing the proof of the theorem.
\end{proof}

\begin{remark}
	Concerning the conditions in the theorems we apply, note that
	\begin{enumerate}
		\item the orbit equivalence relation of $\Gam\act Y$ is hyperfinite and hence not strongly ergodic,
		\item the Polish groups involved (namely $\hat{\Gam}$, $S^1$ and $L(Y,\nu,S^1)$) admit invariant metrics,
		\item $\hat{\Gam}\times L(Y,\nu,S^1)\curvearrowright Z^1(\Gam\act Y,S^1)$ is an action by isometries. Indeed, on $Z^1(\Gam\act Y,S^1)$ we have the compatible metric
		\[
		d(c,c') = \sum_{k=1}^\infty 2^{-k} \int_Y \abs{c(g_k,y)-c'(g_k,y)}\,d\nu(y)
		\]
		where $\Gam:=\{g_1,g_2,g_3,\dots\}$ (see also \cite[ch.~24]{Kec10}). It is then a straightforward calculation to check that $d((\gam,f)\cdot c,(\gam,f)\cdot c')=d(c,c')$.
	\end{enumerate}
\end{remark}

We end with the following proposition, which together with Theorem~\ref{thm:R} easily implies Corollary~\ref{cor:McDuff}.

\begin{proposition}\label{prop:MotbN}
	Suppose $M$ is a \II factor with at least one Cartan subalgebra and $N$ is any \II factor. Then $\cR(\cU(N)\act\Cartan(N))\leq_B \cR(\cU(M\otb N)\act\Cartan(M\otb N))$.
\end{proposition}
\begin{proof}
	Let $A\sub M$ be a Cartan subalgebra and consider the Borel map
	\[
	f: \Cartan(N)\rarrow \Cartan(M\otb N): B\mapsto A\otb B.
	\]
	We claim that this is a Borel reduction for the equivalence relations induced by unitary conjugacy, i.e. $B$ and $C$ are unitarily conjugate inside $N$ if and only if $A\otb B$ and $A\otb C$ are unitarily conjugate inside $M\otb N$. One direction is trivial. For the other direction, suppose that $A\otb B$ is unitarily conjugate to $A\otb C$ inside $M\otb N$.
	\begin{claim}
		If $A\sub M$ and $B,C\sub N$ are any von Neumann algebras such that $A\otb B$ is unitarily conjugate to $A\otb C$ inside $M\otb N$, then $B\emb_N C$.
	\end{claim}
	\begin{proofofclaim}
		Suppose not, then by Theorem~\ref{thm:Popa} there exist $u_n\in\cU(B)$ such that $\norm{E_C(x^*u_ny)}_2\rarrow 0$ for all $x,y\in N$. Take now $s_1,s_2\in M$ and $t_1,t_2\in N$. Then
		\[
		\norm{E_{A\otb C}((s_1\ot t_1)^*(1\ot u_n)(s_2\ot t_2))}_2 = \norm{E_A(s_1^*s_2)\ot E_C(t_1^*u_nt_2)}_2 \rarrow 0.
		\]
		Since elements of the form $s\ot t$ are dense in $M\otb N$, it follows by normality of the conditional expectation that the same holds for any $x,y\in M\otb N$. Hence $A\otb B\not\emb_{M\otb N} A\otb C$, contradiction.
	\end{proofofclaim}
	
	\vspace{3pt}
	
	Since $N$ is a \II factor and $B,C$ are Cartan subalgebras, the claim implies that $B$ and $C$ are unitarily conjugate inside $N$, finishing the proof.
\end{proof}

\begin{proof}[Proof of Corollary~\ref{cor:McDuff}]
	Since a McDuff \II factor $M$ satisfies $M\cong M\otb R$ by definition, this follows immediately from Proposition~\ref{prop:MotbN} and Theorem~\ref{thm:R}.
\end{proof}

\printbibliography[heading=bibintoc]

\Addresses

\end{document}